\documentclass[12pt]{amsart}
\usepackage{amsfonts, amsthm, amsmath, amssymb}
\usepackage{hyperref}
\hypersetup{colorlinks=false}
\usepackage{graphicx}
\usepackage{subcaption}

\usepackage{tikz-cd}

\newtheorem{theo}{Theorem}[section]

\newtheorem{lemma}[theo]{Lemma}

\theoremstyle{remark}
\newtheorem{remark}[theo]{Remark}
\newtheorem{question}[theo]{Question}
\newtheorem{desiderata}[theo]{Desiderata}
\newtheorem{listicle}[theo]{Properties}
\theoremstyle{definition}

\newcommand{\Frob}{\operatorname{Frob}}
\newcommand{\tr}{\operatorname{tr}}

\numberwithin{equation}{section}

\begin{document}

\title[Singularities over function fields]{Singularities and vanishing cycles in number theory over function fields}

\author{Will Sawin}

\maketitle

\begin{abstract} This article is an overview of the vanishing cycles method in number theory over function fields. We first explain how this works in detail in a toy example, and then give three examples which are relevant to current research. The focus will be a general explanation of which sorts of problems this method can be applied to.  \end{abstract}

\section{Introduction}

This article concerns number theory over function fields - the study of classical problems in number theory, transposed to the function field setting. These proposed problems are then amenable to attack by topological, and more precisely cohomological, methods. The goal of this article is to give a broad, conceptual introduction to a specific set of topological techniques that have proven useful to recent works in this area. We begin with a brief review of the general study of function field number theory, though we recommend that anyone unfamiliar with it read another article on the subject (e.g. \cite{Ellenberg-AWS}, \cite{CEF}) before tackling this one 

\subsection{Number theory over function fields} We start by expressing a problem in number theory as the problem of finding an approximate formula (or proving an already conjectured one) for the number of elements in a fixed set. Many problems in analytic number theory and arithmetic statistics have this form, or can be reduced to it. For instance, a solution to the twin primes conjecture on the infinitude of twin primes would follow from a counting formula for the number of twin primes at most $N$, for an arbitrary integer $N$.

If this problem seems intractable, we can consider a finite field analogue, which may be easier. We take the definition of the set we wish to count and modify it by replacing each occurrence of the integers in the definition with the ring $\mathbb F_q[T]$ of polynomials over a finite field, but otherwise change it as little as possible. For instance, we could replace twin primes (numbers $n$ such that $n$ and $n+2$ are prime) with polynomials $f$ such that $f$ and $f+2$ are both irreducible polynomials.  (More generally, if our problem involves the ring of integers of a number field, we can replace it with the ring of functions on an affine algebraic curve over $\mathbb F_q$.)

We can almost always recognize this new set as the set of points of an algebraic variety defined over $\mathbb F_q$, i.e. the solutions in $\mathbb F_q$ of some finite list of polynomial equations, possibly after removing the solutions to another list of polynomial equations. The Grothendieck-Lefschetz fixed-point formula then gives an approximate formula for the number of points in this new set if we can calculate the (high-degree, compactly-supported, \'{e}tale) cohomology of this algebraic variety.

If we don't know how to calculate this cohomology, we could transfer the problem to yet another setting. We can consider the analogous, purely topological problem, of calculating the (high-degree, compactly-supported, singular) cohomology of the space of solutions in $\mathbb C$ to the same list of polynomial equations. This can be attacked using any of the many known topological methods for computing cohomology groups. Once this is done, we can sometimes transfer back to the original finite field problem using comparison results for cohomology, or by using finite field analogues of our topological techniques. However, it is rarely possible to go all the way back and solve the original number theory problem by this approach.

\subsection{The topological methods} This article will primarily discuss a cluster of methods useful in the last step - the application of topological tools to calculate cohomology groups. This cluster centers around vanishing cycles, and depending on the problem may also involve perverse sheaves, the characteristic cycle, and the search for suitable compactifications. We will focus mainly on vanishing cycles, but mention the others, because they seem to be closely related. These methods arose mainly from the study of the topology of algebraic varieties over the complex numbers - for instance vanishing cycles arose in the work of Lefschetz, and perverse sheaves arose in the work of Goresky and Macpherson \cite{GoreskyMacpherson}. They were then later applied to the study of varieties over finite fields as well - vanishing cycles by the SGA authors \cite{sga7-i}, and perverse sheaves by Beilinson, Bernstein, Deligne, and Gabber \cite{BBDG}.

\subsection{Comparison with homological stabilization theory} Before discussing these tools, we will compare them to homological stabilization theory - a different set of tools that can be used to tackle similar problems. In particular, the brilliant papers \cite{EVW} and \cite{EWT} reduced two number-theoretic problems over function fields to calculating the cohomology of certain Hurwitz spaces, and used tools from stable cohomology to make great progress on understanding, though not to completely compute, these cohomology groups.

In view of the philosophy that the most helpful information a paper can contain is an explanation of when the other information in the paper will be useful, let us discuss the situations when these two methods can be applied. Homological stabilization theory seems to be most effective when the space $X$ we wish to study have the following properties:

\begin{desiderata} \label{intro-hst-desiderata}\phantom{.}

\begin{enumerate}

\item $X$ is a manifold. In particular, we can use Poincar\'{e} duality to express the (high-degree) compactly supported cohomology of $X$ in terms of the (low-degree) usual cohomology.

\item $X$ has a nice descriptions as a classifying space of groups. For instance, $X$ may be a configuration space, a moduli space of curves, or a cover of one of these.

\end{enumerate}
\end{desiderata}

For example, the Hurwitz spaces studied in \cite{EVW} and \cite{EWT} are certainly manifolds, and moreover are covering spaces of configuration spaces. Thus, they (or their connected components) can be described as the classifying spaces of certain explicit subgroups of the braid group. On the other hand, the spaces discussed in this paper share a contrasting set of properties:

\begin{desiderata}\label{intro-vc-desiderata} \phantom{.}

\begin{enumerate}

\item $X$ is not a manifold - instead, it has singularities.

\item $X$ has a nice description as the set of solutions of a system of polynomial equations (e.g. it is described by simple equations, or equations of low degree, or relatively few equations, or all of these at once.)

\item $X$ does not have any obvious nice descriptions in terms of configuration spaces, moduli spaces of curves, etc. 

\item $X$  lies in a natural family of spaces $X_t$, depending on some auxiliary parameters $t$. All the spaces in this family are defined by similar equations, but may have very different singularities. A generic member of the family often has no singularities at all.

\end{enumerate}

\end{desiderata}

Because homological stabilization was developed to study the usual cohomology of classifying spaces of groups, and is best adapted for this situation, conditions (1) and (3) are problematic if we wish to apply methods from that field.

More subtly, condition (4) is also problematic for methods from the field of homological stabilization or other purely topological approaches. If our spaces lie in a parametric family of different spaces, which are not all homeomorphic, then they are unlikely to have any nice purely topological description, even a non-obvious one. This is because topological constructions usually don't depend on continuous auxiliary parameters in a nontrivial way. Instead, topological objects are invariant under continuous deformations.

On the other hand, for the methods discussed in this paper, not being a classifying space is not a problem, singularities can be dealt with, and (2) and (4) are both very helpful. 

\subsection{Contents} This article will begin, in Section \ref{s-vc}, with a conceptual introduction to vanishing cycles theory that examines how it can be applied to a simple toy problem. (Some introductions with a greater level of technical detail include, in complex geometry, \cite{MasseyIntro}, and in the $\ell$-adic setting, \cite[\S3]{FreitagKiehl} and \cite[\S9.2]{Fu}.) We will then move on to discussing three specific problems from recent research. We will explain the number-theoretic motivation, the spaces $X$ whose topology must be studied, and how they are related. We will then discuss how the spaces $X$ realize the properties (1)-(4) and how vanishing cycles, and related methods, have been used to attack them.

In number-theoretic terms, the problems we discuss will be
\begin{itemize}

\item The moments of $L$-functions. (Section \ref{s-lf})

\item The values of character sums over intervals. (Section \ref{s-cs})

\item The distribution of CM points on $X(1) \times X(1)$. (Section \ref{s-st})

\end{itemize}
We do not assume familiarity with these problems in their classical, finite field, or topological incarnations.  

\subsection{Conclusion} A secondary purpose of this article is to be a survey of the great diversity of spaces whose cohomology is of interest in analytic number theory. In particular, many seem to arise from different fields of math. Some, such as the Hurwitz spaces, are perhaps most naturally defined topologically. Others, such as spaces of tuples of polynomials satisfying a congruence condition (discussed in \S\ref{s-lf}), clearly hail from the world of algebraic geometry. Spaces defined using the Jacobian variety and Abel-Jacobi map (see \S\ref{s-st}), might be said to appear most naturally in complex geometry. The complements of hyperplane arrangements (\S\ref{s-cs})  are quite natural from the perspective of all these fields. It seems reasonable to believe that the greatest progress in function field number theory in the future will come from combining vanishing cycles techniques, other topological techniques, and number-theoretic techniques to achieve something none of them could alone.

\subsection{Acknowledgments} 
 This article was written while the author served as a Clay Research Fellow. I would like to thank Johan de Jong, Jordan Ellenberg, Philippe Michel, and (especially) the anonymous referee for many helpful comments on earlier versions of this paper.

\section{Vanishing cycles}\label{s-vc}

Vanishing cycles originally arose in the context where we have a map of algebraic varieties over the complex numbers $f: X \to \mathbb A^1_{\mathbb C}$ (or in topological language $f: X \to \mathbb C$) and view the inverse images $X_t$ of $t \in \mathbb A^1_{\mathbb C}$ (or $t\in \mathbb C$) as a parametric family of varieties. When $f$ is smooth and proper, the fibers are compact manifolds, and moreover are all homeomorphic because we can flow one fiber into another using Ehresmann's lemma. Thus, they have the same cohomology. Vanishing cycles theory is designed to study how the cohomology differs when this does not hold, and in particular the fibers are not all smooth.

\subsection{A family of elliptic curves}\label{ss-elliptic} To see how this theory works, let us consider the family of elliptic curves $E_t$ in $\mathbb P^2(\mathbb C)$ with affine equation \begin{equation}\label{family-of-elliptic} y^2 =x^3 + x^2 + t .\end{equation}
\begin{remark} Taking the solutions to this equation in the projective plane $\mathbb P^2(\mathbb C)$ involves considering the homogenized equation $y^2z = x^3 + x^2 z + tz^3$. Here we have added powers of $z$ to each term so that they all have the maximal degree. We then consider the space of solutions $(x,y,z) \in \mathbb C^3$ of this equation, with $x,y,z$ not all $0$, up to scaling $(x,y,z) \mapsto (\lambda x, \lambda y, \lambda z)$. 

This has the effect of adding to the original solutions, which correspond to triples $(x,y,z)$ with $z \neq 0$, a new point $x=0, y=1, z=0,$ ``at infinity". For general algebraic varieties, we may need to add more than one point at infinity, when there are multiple solutions with $z=0$ of the defining equations.

\end{remark} 

\begin{remark} We have chosen the examples in this section so that nothing interesting happens with the points at infinity, and they can safely be ignored, but dealing with these points can be an important difficulty in research problems.\end{remark}

The derivatives of the polynomial $x^3 + x^2 + t-y^2$ in the $x$ and $y$ variables respectively are $3x^2+2x$ and $2y$. The only points in the plane where these both vanish are $(x,y)=(0,0)$ and $(-2/3,0)$. Thus, the only points where this polynomial and its derivatives all vanish are $(x,y,t)=(0,0,0)$ and $(-2/3,0, -8/27)$. Thus by the implicit function theorem, for $t \neq 0, -8/27$, the solutions of this polynomial form a smooth manifold of complex dimension one - a Riemann surface.  Because this equation has degree $3$, the surface $E_t$ is an elliptic curve, i.e. a Riemann surface of genus one. (We will verify this formally in \eqref{elliptic-cohomology-verification}.) Therefore we can calculate the singular cohomology groups
\begin{equation}\label{elliptic-cohomology} H^0(E_t, \mathbb Q)=\mathbb Q, H^1(E_t, \mathbb Q)= \mathbb Q^2, H^2(E_t, \mathbb Q)= \mathbb Q\end{equation}
for $t \neq 0,-8/27$ (i.e. for \emph{generic} values of $t$).

On the other hand, for $t=0$, the equation \eqref{family-of-elliptic} has a singularity at the point $(0,0)$, and is not smooth. To understand its topology, we can use its parameterization - solutions can be written $x=w^2-1, y= w(w^2-1)$ for $w \in \mathbb C$. In fact, projective solutions to \eqref{family-of-elliptic} can be given by the same formula for $w$ in the Riemann sphere.

\begin{remark} To find this parameterization, one can set $w=y/x$ and then solve for $x$ and $y$ in terms of $w$.\end{remark}

Because $w=y/x$, this parameterization is one-to-one, except for when $x$ and $y$ are both $0$, which occurs for $w=1$ and $w=-1$. It follows that the solution set to \eqref{family-of-elliptic} with $t=0$ is obtained from the Riemann sphere by gluing together the two points $w=1$ and $w=-1$. From this description it is easy to compute its cohomology as \begin{equation}\label{special-fiber-cohomology} H^0(E_0, \mathbb Q)=\mathbb Q, H^1(E_0, \mathbb Q)= \mathbb Q, H^2(E_0, \mathbb Q)= \mathbb Q\end{equation} (perhaps most easily by giving it a CW complex structure with one $0$-cell, one $1$-cell, and one $2$-cell.)

The vanishing cycles method explains the discrepancy between \eqref{elliptic-cohomology} and \eqref{special-fiber-cohomology} using local information at the singular point $(0,0)$.

Given any cycle in the homology of $E_t$, as $t$ goes to $0$ we can continually push the cycle into a new fiber, producing a cycle in $E_0$. If $E_0$ were smooth, this map would give an isomorphism on homology. However, where there is a singularity something else can happen - a nontrivial cycle in $E_t$ can collapse down to a single point in $E_0$. This happens precisely for one nontrivial circle around the torus. In fact, taking a circle on a torus and collapsing it down to a single point produces a space which is homeomorphic to a sphere with two points glued together. as in Figure \ref{torus-figure}. \begin{figure} \includegraphics[width=0.25\linewidth]{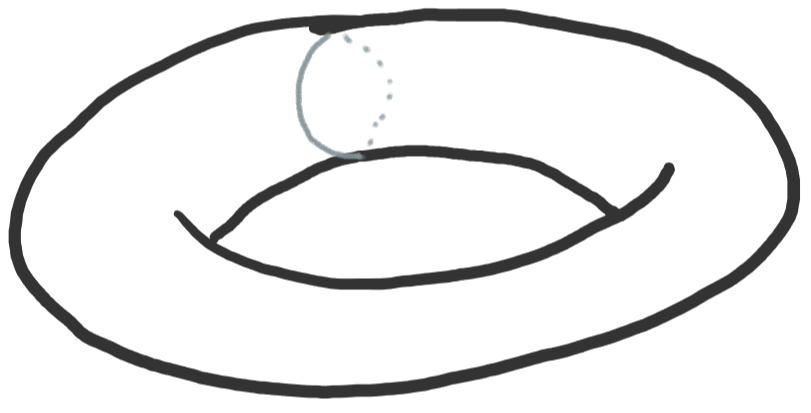} \includegraphics[width=0.25\linewidth]{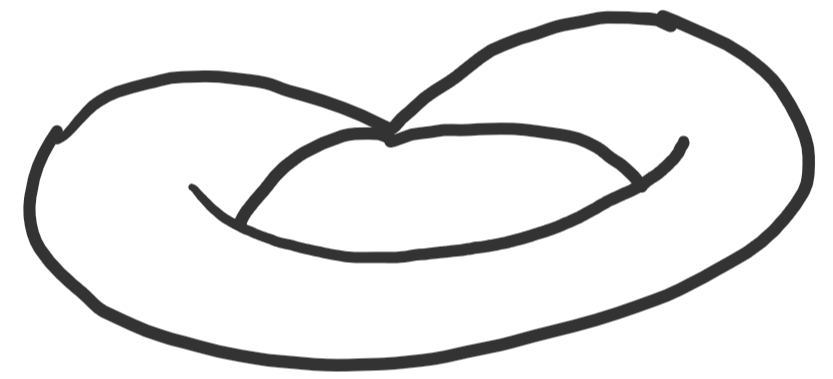} \caption{The circle on the left contracts to the singularity on the right.}\label{torus-figure}  \end{figure}

\subsection{Generalities on vanishing cycles} More generally, the classical vanishing cycle theory shows, when $X_t$ is a parametric family of complex spaces where $X_t$ is smooth for all but finitely many $t$ and $X_0$ has isolated singularities, that:

\begin{itemize}
\item  There is a natural map $H_* ( X_t) \to H_* (X_0)$, for generic $t$ near $0$.
\item This map fits into a long exact sequence, whose third term is a sum of contributions from singular points.
\item The contribution of each singular point depends only on local data in a neighborhood of this point - more precisely, it is the reduced homology of the Milnor fiber. 
\end{itemize}

In the arithmetic setting, it is most convenient to express this data using sheaf cohomology:

\begin{listicle} \label{properties-of-vanishing-cycles} \phantom{.}

\begin{itemize}

 \item There exists a map $H^* (X_0,\mathbb Q) \to H^* (X_\eta, \mathbb Q)$, for $\eta$ the generic value of $t$.
 
\item This map fits into a long exact sequence, whose third term is the cohomology $H^*(X_0, R\Phi \mathbb Q)$ on $X_0$ of a complex of vanishing cycles $R\Phi \mathbb Q$.

\item $R\Phi \mathbb Q$ is supported at the singularities of $X_0$.

\item $R\Phi \mathbb Q$ is defined locally, so its stalk at a point depends only on the geometry of $X$ and $f$ in a neighborhood of that point.

\end{itemize}

\end{listicle}

 If the singularities of $X_0$ are isolated, then $H^i(X_0, R\Phi \mathbb Q)$ is by definition the sum over the singular points of $X_0$ of the stalk of $R^i\Phi \mathbb Q$ at that point, and this recovers the cohomological dual of the classical picture. If there are higher-dimensional singularities, then cohomology of the complex $R\Phi \mathbb Q$ is a more complicated operation.

For now, let us focus on the case where singularities are isolated. In particular, we will show how, using only these facts, and not knowing anything else about the definition of vanishing cycles or elliptic curves, one can recover the cohomology of the family $E_t$ described by \eqref{family-of-elliptic}. This fits with how vanishing cycles arguments in the arithmetic setting usually proceed in practice - it's usually most convenient to use abstract properties as much as possible and definitions and constructions as little as possible.

\begin{remark} For precision, we will state the definition of vanishing cycles in a general context. Let $X$ be a variety with a map $f: X \to \mathbb A^1_{\mathbb C}$. Let $X_0$ be the fiber of $X$ over $0$ and let $i: X_0 \to X$ be the inclusion. Let $D \subset \mathbb A^1_{\mathbb C}$ be a small disc around $0$. Let $\tilde{D}$ be the universal cover of the punctured disc $D \setminus \{0\}$. Let $\tilde{X}  = X \times_{\mathbb A^1_{\mathbb C}} \tilde{D}$ and let $p: \tilde{X} \to X$ be the projection map.

There are pullback and derived pushforward functors, $p^*$ and $Rp_*$ respectively, associated to $p$, and an adjunction map $\mathbb Q \to R p_* p^* \mathbb Q$. Applying the pullback functor $i^*$ we obtain a map $i^* \mathbb Q \to i^* R p_* p^* \mathbb Q$, whose mapping cone is defined to be $R \Phi \mathbb Q$.

From this definition the claims of Properties \ref{properties-of-vanishing-cycles} can be deduced: From the definition of the mapping cone we have a long exact sequence \[H^* (X_0, i^*\mathbb Q) \to H^* (X_0 ,  i^* R p_* p^* \mathbb Q)\to H^*( X_0, R\Phi \mathbb Q).\] We have $i^* \mathbb Q= \mathbb Q$ so the first term is $H^*(X_0,\mathbb Q)$, and one can check that the second term is $H^* (X_\eta, \mathbb Q)$ using the proper base change theorem. Because of the smooth base change theorem, $R \Phi \mathbb Q$ vanishes away from singularities, and because both $R p_*$ and the adjunction can be defined locally, the vanishing cycles can be defined locally.

See \cite{MasseyIntro} or \cite[XIII - XVI]{sga7-ii}  for more on this.\end{remark}

\subsection{A family of conics}\label{ss-conics} To understand the cohomology of the family $E_t$ from \S\ref{ss-elliptic}, we will compare $E_t$ to a simpler family $X_t$, defined by the equation \begin{equation}\label{quadric-family} y^2= x^2 + t \end{equation} or equivalently \begin{equation} (y+x) (y-x) = t .\end{equation} Homogenizing this equation, we get $y^2 = x^2 + tz^2$, which when $z=0$ has two nonzero solutions, $(x,y,z)=(1,1,0)$ and $(1,-1,0)$, up to scaling. Thus the projectivization of this family has two additional points at infinity.

For $t\neq 0$, we can parameterize the solutions of this equation by setting $w=y+x$, so $y-x = \frac{t}{w}$ and thus \[x = \frac{w}{2} - \frac{t}{2w}, y = \frac{w}{2} + \frac{t}{2w}.\] Here $w=0$ and $w =\infty$ go to the two points at infinity. This map shows that the solution set is isomorphic to the Riemann sphere, and thus \begin{equation}\label{quadric-cohomology} H^0(X_t, \mathbb Q) = \mathbb Q, H^1(X_t, \mathbb Q)=0, H^2(X_t, \mathbb Q) = \mathbb Q \end{equation} for $t\neq 0$. On the other hand, if $t=0$, the solution set consists of two Riemann spheres, with $x+y=0$ and $x-y=0$ respectively, glued at the point $(x,y)=(0,0)$. The cohomology of two spheres glued at a point can be calculated as \begin{equation}\label{quadric-singularity} H^0(X_0, \mathbb Q) = \mathbb Q, H^1(X_0, \mathbb Q)=0, H^2(X_0, \mathbb Q) = \mathbb Q^2 \end{equation} by using a CW complex with one $0$-cell and two $2$-cells. 

We now consider the vanishing cycles long exact sequence associated to this family, for $\eta$ a generic value of $t$: 
\[ \begin{tikzcd} H^0(X_0, \mathbb Q)\arrow[r] & H^0(X_\eta,\mathbb Q) \arrow[r] & H^0(X_0 , R \Phi \mathbb Q) \arrow[dll] \\
 H^1(X_0, \mathbb Q)\arrow[r] & H^1(X_\eta,\mathbb Q) \arrow[r] & H^1(X_0 , R \Phi \mathbb Q) \arrow[dll] \\
  H^2(X_0, \mathbb Q)\arrow[r] & H^2(X_\eta,\mathbb Q) \arrow[r] & H^2(X_0 , R \Phi \mathbb Q) \end{tikzcd}\]
  which specializes to
  \[ \begin{tikzcd} \mathbb Q\arrow[r] &\mathbb Q \arrow[r] & H^0(X_0 , R \Phi \mathbb Q) \arrow[dll] \\
 0\arrow[r] &0 \arrow[r] & H^1(X_0 , R \Phi \mathbb Q) \arrow[dll] \\
\mathbb Q^2\arrow[r] & \mathbb Q \arrow[r] & H^2(X_0 , R \Phi \mathbb Q) \end{tikzcd}\]
To calculate this long exact sequence completely, it suffices to check that the maps $H^0(X_0 , \mathbb Q) \to H^0( X_{\eta}, \mathbb Q)$ and $H^2(X_0 , \mathbb Q) \to H^2( X_{\eta}, \mathbb Q)$ are surjective. In the classical setting, we would check that flowing the corresponding homology cycles - either a point or all of $X_t$ - from $X_t$ to $X_0$ produces a nonzero cycle, in this case either a point or all of $X_0$. In an arithmetic setting, it might be more straightforward to check from general vanishing results that $R\Phi \mathbb Q$ vanishes in degrees $-1$ and $2$, so  $H^{-1}(X_0 ,R\Phi \mathbb Q) = H^2(X_0 , \mathbb R \Phi \mathbb Q)=0$. 

From this long exact sequence, we now obtain \[ H^0( X_0, R\Phi\mathbb Q) = 0 , H^1(X_0, R \Phi \mathbb Q) =\mathbb Q, H^2(X_0 , \mathbb R \Phi \mathbb Q)=0.\] Recalling that in the case of isolated singularities, this cohomology is a fancy notation for a sum over the singular points, we conclude that the contribution of the singularity $(0,0)$ is $\mathbb Q$ in degree $0$. 

\subsection{Comparing the two families} Now comes the key point: A neighborhood of the singularity $x=0, y=0, t=0$ in the solution set of the equation \eqref{quadric-family} is isomorphic to a neighborhood of the singularity $x=0,y=0,t=0$ of the equation \eqref{family-of-elliptic}, under the map that fixes $y$ and $t$ and sends $x$ to $x \sqrt{1+x} = x + x^2/ 2 - x^3/ 4 + 3 x^4/8 - \dots $. Thus, the vanishing cycles contributions of that point to the two different families must be identical.  Furthermore, in $E_0$ there are no other singularities that contribute, so we must also have  \[ H^0( E_0, R\Phi\mathbb Q) = 0 , H^1(E_0, R \Phi \mathbb Q) =\mathbb Q, H^2(E_0 ,  R \Phi \mathbb Q)=0.\] 
Thus, the vanishing cycles long exact sequence
\[ \begin{tikzcd} H^0(E_0, \mathbb Q)\arrow[r] & H^0(E_\eta,\mathbb Q) \arrow[r] & H^0(E_0 , R \Phi \mathbb Q) \arrow[dll] \\
 H^1(E_0, \mathbb Q)\arrow[r] & H^1(E_\eta,\mathbb Q) \arrow[r] & H^1(E_0 , R \Phi \mathbb Q) \arrow[dll] \\
  H^2(E_0, \mathbb Q)\arrow[r] & H^2(E_\eta,\mathbb Q) \arrow[r] & H^2(E_0 , R \Phi \mathbb Q) \end{tikzcd}\]
using this equation and the calculation \eqref{special-fiber-cohomology} of the cohomology of the special fiber, reduces to
\[ \begin{tikzcd} \mathbb Q\arrow[r] & H^0(E_\eta,\mathbb Q) \arrow[r] & 0  \arrow[dll] \\
\mathbb Q\arrow[r] & H^1(E_\eta,\mathbb Q) \arrow[r] & \mathbb Q \arrow[dll] \\
  \mathbb Q\arrow[r] & H^2(E_\eta,\mathbb Q) \arrow[r] & 0  \end{tikzcd}\] 
Using the the fact that $E_\eta$ is a compact surface and so $H^2(E_\eta,\mathbb Q) \neq 0$, we can see that the map into $H^2(E_\eta,\mathbb Q)$ is nonzero, which determines all the other maps and in particular forces \begin{equation}\label{elliptic-cohomology-verification} H^0 (E_\eta, \mathbb Q) = \mathbb Q, H^1(E_\eta, \mathbb Q) = \mathbb Q^2, H^2(E_\eta, \mathbb Q) = \mathbb Q.\end{equation} Thus, we have verified \eqref{elliptic-cohomology}. By the classification of oriented surfaces, we can confirm that $E_\eta$ must be a surface of genus one.

Of course, there are easier ways to check that $E_t$ is topologically a torus. But there are other problems where the vanishing cycles strategy demonstrated here is essential. We can find a nice family of spaces, and try to deduce facts about a particular fiber from knowledge about the general fiber, or, as in this case, deduce facts about the generic fiber from information about a particular one. It suffices to find all singularities of this particular fiber and estimate their vanishing cycles contribution. To do this we can look at a small neighborhood of the singularity and try to understand its geometry. In particular, we could find another family of spaces, defined by a simpler equation, with a singularity of the same type - ideally as its only singularity.  Finally, we can reduce the vanishing cycles question to a global question about this simpler space. Hopefully, we can solve this problem directly. By reducing a global question on our original family to a local question at a singularity to a global question on a new family, we have progressively simplified the problem, eventually reaching a form that we know how to solve.

\subsection{Calculating the monodromy}

In addition to calculating the cohomology of the fibers $E_t$, we can gain a little more information about the family $E_t$ using the same analysis of the singularity we did already. Suppose we send $t$ in a small circle around $0$, and let $E_t$ vary. Because $E_t$ remains smooth for all these values, we can always flow $E_t$ into the next one. This flow defines a mapping class of homeomorphisms $E_t \to E_t$ and thus a map on cohomology $H^*(E_t, \mathbb Q) \to H^*(E_t , \mathbb Q)$. This map is known as the monodromy action. Because it comes from an (orientation-preserving) homeomorphism, it certainly fixes $H^0$ and $H^2$. How does it act on $H^1$?

It is a classical fact that, for this family and similar ones, the action is unipotent. Topologically, one can recognize this mapping class as a Dehn twist, which always acts by unipotent transvections on $H^1$. A more number-theoretic argument would recognize this loop as a loop around a cusp in the modular curve $X(1)$ parameterizing elliptic curves. The modular curve is a quotient of the upper half-plane by $SL_2(\mathbb Z)$, with each loop defining a conjugacy class in $SL_2(\mathbb Z)$, and the action on $H^1$ given by the standard representation of $SL_2(\mathbb Z)$. From examining the geometry of the upper half-plane and this action, one sees that a loop around a cusp defines a unipotent conjugacy class.

Let us see how we can verify this same fact with vanishing cycles. The key facts are that 

\begin{itemize}

\item We can define a monodromy action on $H^*(X_0, R\Phi \mathbb Q)$ such that all the maps of the vanishing cycles long exact sequence commute with monodromy. Here we take monodromy to act trivially on $H^*(X_0,  \mathbb Q)$.

\item This action on the contribution of a singularity depends only on the geometry of $X$ and $f$ in a neighborhood of that singularity. \end{itemize}

Examining the vanishing cycles long exact sequence for the family of conics $X_t$ defined in \S\ref{ss-conics}, we see that $H^1(X_0 , R \Phi \mathbb Q) $ is a subspace of $H^2(X_0, \mathbb Q)$. Hence, because it is a subspace of a space with trivial monodromy action, it has trivial monodromy action. Thus, in a short exact piece\[ 0\to H^1(E_0, \mathbb Q) \to H^1(E_\eta, \mathbb Q) \to H^1(E_0, R\Phi \mathbb Q) \to 0\] of the vanishing cycles long exact sequence for $E_t$, which reduces to \[ 0 \to \mathbb Q\to H^1(E_\eta, \mathbb Q) \to \mathbb Q\to 0,\] the monodromy action is trivial on both $\mathbb Q$s. Hence the action on $H^1(E_\eta, \mathbb Q)$ must be unipotent.

So we have recovered this unipotence phenomenon without needing any understanding of the surface $E_t$. One could go further, and check that this monodromy action is nontrivial unipotent, using Deligne's theory of weights, but this would take us too far afield.

\subsection{Conclusion}

The arguments we have just seen highlight some key properties of vanishing cycles theory, which should start to provide an explanation for Desiderata \ref{intro-vc-desiderata}:

\begin{enumerate}

\item The singularities of the special fibers $E_0$ and $X_0$ did not harm our argument. Instead, they helped - the singularity of $E_0$ gave it a rational parameterization, making its cohomology easier to compute, which we could use, together with the family $X_t$, to compute the cohomology of the more difficult smooth fiber $E_t$.

\item The simplicity of the polynomial equation defining $E_t$ was a big help. It let us find the singularity, and once the singularity was found, suggested how to simplify the geometry by removing the few high-degree terms in the equation.

\item The fibers $E_t$ have a nice description as topological tori, and the family $E_t$ has a nice description as a cover of the moduli space of elliptic curves, but we did not use either of these.

\item It was absolutely crucial that we were studying a family $E_t$, and not just a single Riemann surface. The fact that the generic member was smooth was also helpful.

\end{enumerate}

Let us also mention two more technical points about using vanishing cycles in general:

\begin{remark} Vanishing cycles theory is simplest when the singularities are isolated, i.e. the singular locus has dimension $0$. But even when the singularities are not isolated, it is better for the dimension of the singular locus to be smaller. This is especially crucial because of what we know about the degree of the vanishing cycle complex, which in general is a complex of sheaves and not just a sheaf. It was proven in \cite[I, Corollary 4.3]{sga7-i} that $R \Phi \mathbb Q$ has a property that was later called semiperversity - its $q$'th cohomology sheaf $\mathcal H^q(  R \Phi \mathbb Q)$ is supported in a set of dimension $\dim X_0 - q$.  Because any sheaf on an algebraic variety of dimension $n$ has cohomology in degree at most $2n$, this implies that $H^p ( \mathcal H^q( R \Phi \mathbb Q))$ vanishes unless \[ p \leq 2 ( \min ( \dim X_0-q, \dim ( \operatorname{support} ( R \Phi \mathbb Q) ))) \leq   \dim X_0- q + \dim ( \operatorname{support} ( R \Phi \mathbb Q) )),\]
which combined with the spectral sequence \[H^p ( \mathcal H^q( R \Phi \mathbb Q)) \mapsto H^{p+q} (X_0 , R\Phi \mathbb Q)\] implies that $H^i (X_0, R\Phi \mathbb Q)$ vanishes unless \[i \leq \dim X_0 + \dim ( \operatorname{support} ( R \Phi \mathbb Q) ) \leq \dim X_0 + \dim X_0^{\operatorname{Sing}},\] where $X_0^{\operatorname{Sing}}$ is the singular locus of $X_0$. This vanishing was used to great effect in \cite[appendix by Katz]{Hoo91}. Even when vanishing is not enough and more precise control of the vanishing cycles is needed, this is generally easier to achieve when the dimension of the singular locus is smaller. Finding an interesting space whose singular locus is unexpectedly small can itself represent substantial progress towards the solution of an arithmetic problem. \end{remark}

\begin{remark} It is required for the vanishing cycles exact sequence to hold that the map $f: X \to \mathbb C$ be proper, as otherwise a cycle could escape off to infinity, which would not be explained by any data locally in the fiber. Thus, to apply this method to noncompact spaces, it is necessary to compactify by adding additional points, usually taking the form of a ``boundary divisor" or ``divisor at infinity". Then one has to control the vanishing cycles for points at this divisor over the special fiber as well, and it is crucial to choose a compactification where this can be done. \end{remark} 

\section{Statistics of $L$-functions}\label{s-lf}

We begin this section with a brief precise statement of the question we will consider, before explaining in more detail the motivation behind it in \S\ref{ss-lf-bm}, and then finally studying it as a topological problem in \S\ref{ss-lf-C}. We will see that vanishing cycles can solve important special cases of this problem, but not yet the full problem - something that future work will hopefully be able to rectify.

Let $K$ be a field (either finite or the complex numbers), $g$ a polynomial of degree $m$ over $K$, $k$ a natural number, and $d_1,\dots,d_{2k}$ additional natural numbers.

Let $X_{(d_i)_{i=1}^{2k}, g}(K)$ be the set of tuples $(f_1,\dots, f_{2k})$ of polynomials over $K$, with $f_i$ monic of degree $d_i$, where $f_i$ and $g$ share no common factors, and where \[\prod_{i=1}^k f_i \equiv \prod_{i=k+1}^{2k} f_i \mod g.\]

Some natural number-theoretic problems over $\mathbb F_q[T]$ reduce to counting the elements in $X_{(d_i)_{i=1}^{2k}, g}(\mathbb F_q)$. This motivates us to consider the analogous question of calculating the compactly-supported cohomology of $X_{(d_i)_{i=1}^{2k}, g}(\mathbb C)$ (given its natural topology as a subspace of the space $\mathbb C^{ \sum_{i=1}^{2k} d_i}$ parameterizing tuples of monic polynomials of degree $d_i$).

\subsection{Background and motivation}\label{ss-lf-bm}

Fix a finite field $\mathbb F_q$ and polynomial ring $\mathbb F_q[T]$ in one variable over that finite field. Let $g \in \mathbb F_q[T]$ be a polynomial. We have the quotient ring $\mathbb F_q[T]/g$ and its group of invertible elements $\left( \mathbb F_q[T]/g\right)^\times$. The class in $\mathbb F_q[T]/g$ of a polynomial $f \in \mathbb F_q[t]$ is invertible if and only if $\gcd(f,g)=1$.

For $\chi: ( \mathbb F_q[T]/g )^\times \to \mathbb C^\times$ a character, we can form the Dirichlet $L$-function
\[ L(\chi,u) = \sum_{d=0}^{\infty} \Bigl( \sum_{ \substack { f \in \mathbb F_q[T] \\ \textrm{monic} \\ \textrm{degree }d \\ \gcd(f,g)=1}} \chi(f) \Bigr) u^d .\]

This is a close analogue of the usual Dirichlet $L$-function studied over the integers, which for $m \in \mathbb Z$ and $\chi: (\mathbb Z/m)^\times \to \mathbb C^\times$ is

\[ L(\chi,s) = \sum_{ \substack {n \in \mathbb Z \\ \textrm{positive} \\ \gcd(n,m)=1}} \chi(n) n^{-s} .\]

Here we think of the positive condition and monic condition as being analogous, and the $u^d$ terms and $n^{-s}$ terms as being analogous.

\begin{remark} We can make the second analogy more clear by setting $u=q^{-s}$, so $u^{\deg f} = q^{-s \deg f} = \left(q^{\deg f}  \right)^{-s}$, and noting that for an integer $n$, $|n|$ is the cardinality of the quotient ring $\mathbb Z/n$, so we should define $|f|$ to be the cardinality of the quotient ring $\mathbb F_q[T]/f$, which is $q^{\deg f}$. However, substituting $q^{-s}$ for $u$ would make our formulas slightly more complicated, so we avoid it for the rest of this argument. \end{remark}

A key fact about $L(\chi,u)$ is that it is almost always a polynomial in $u$:

\begin{lemma} Fix $g \in \mathbb F_q[T]$ of degree $m$ and $\chi: ( \mathbb F_q[T]/g )^\times \to \mathbb C^\times$ a nontrivial character. Then $L(\chi,u)$ is a polynomial in $u$ of degree $\leq m-1 $.\end{lemma}

\begin{proof} It suffices to prove that the coefficient \[\sum_{ \substack { f \in \mathbb F_q[T] \\ \textrm{monic} \\ \textrm{degree }d \\ \gcd(f,g)=1}} \chi(f) \] of $u^d$ in $L(\chi,u)$ vanishes for $d \geq m$. To do this, note that when $d \geq m$, any residue class in $\mathbb F_q[t]/g$ occurs for exactly $q^{d-m}$ monic polynomials $f$ of degree $d$. Indeed we can represent any class in $\mathbb F_q[t]/g$ by a polynomial of degree $m-1$, and then add the product of $g$ by any monic polynomial of degree $d-m$ to find such an $f$. Because there are $q^{d-m}$ monic polynomials of degree $d$, we have $q^{d-m}$ choices for $f$. So the coefficient of $u^d$ in $L(\chi,u)$ equals
\[ q^{d-m} \sum_{ \substack { f \in (\mathbb F_q[T]/g )^\times }} \chi(f).\]

However, \[ \sum_{ \substack { f \in (\mathbb F_q[T]/g )^\times }} \chi(f)= 0 \] because the sum of any nontrivial character over the elements of a group vanishes. So the coefficient equals zero, as desired. \end{proof}

In fact, Weil proved \cite{Weil}: 

\begin{theo}\label{weil} Fix $g \in \mathbb F_q[T]$ of degree $m$ and $\chi: ( \mathbb F_q[T]/g )^\times \to \mathbb C^\times$ a character. Assume that

\begin{enumerate}

\item There does not exist a polynomial $g'$ of degree $<m$ dividing $g$ such that $\chi(f)$ depends only on the congruence class of $f$ modulo $g'$.

\item The restriction of $\chi$ to $\mathbb F_q^\times \subseteq  (\mathbb F_q[T]/g )^{\times}$ is a nontrivial character of $\mathbb F_q^\times$. 

\end{enumerate}

Then $L(\chi,u)$ is a polynomial in $u$ of degree $m-1$, with all roots $\alpha \in \mathbb C$ satisfying $|\alpha| =q^{-1/2}$. 

\end{theo}

\begin{remark} Weil proved that the roots satisfy $|\alpha| =q^{-1/2}$ as a special case of his Riemann hypothesis for curves over finite fields. The strategy is similar to the classical argument in number theory that writes the Dedekind zeta function of a cyclotomic field as a product of Dirichlet $L$-functions, so that the Riemann hypothesis for zeta functions of cyclotomic fields would imply the Riemann hypothesis for $L$-functions of characters. The function fields of the curves used in this proof are analogues of the cyclotomic fields, which can be defined using the Carlitz module or using explicit equations similar to \eqref{magic-Lang-formula} below.\end{remark}

\begin{remark} Under assumption (1), we say that $\chi$ is \emph{primitive}, and under assumption (2), we say that $\chi$ is \emph{odd}. \end{remark}

\begin{remark}When the character is not primitive and odd, there is always a simple way to modify $L(\chi,u)$ to obtain a polynomial, of degree $<m-1$, with all roots $\alpha$ satisfying $|\alpha| = q^{-1/2}$, by removing certain \emph{trivial zeroes} which are roots of absolute value $q^{-1}$.

When the character is not primitive, this can be achieved by simply replacing $g$ with the polynomial $g'$ discussed in assumption (1), which is why it is usually fine to restrict attention to the case of primitive characters.\end{remark}

Theorem \ref{weil} is expected to be essentially the strongest result that is true about all $L(\chi,u)$ - in other words, any polynomial of degree $m-1$ with all roots on the circle of radius $q^{-1/2}$ could well appear as some $L(\chi,u)$. (If we allow the parameter $q$ to go to $\infty$, a precise version of this expectation is known to be true in the case where $g$ is squarefree \cite{Katz-PDC}.) Thus, when we study these $L$-functions after Weil, we usually study them in a \emph{statistical} manner - we wish to know how the roots or values of this $L$-function behave, as random variables, when we select $\chi$ uniformly at random from the set of all characters (or all primitive odd characters - we will ignore this technical distinction.)

One fundamental statistical question about the $L$-function are its moments, or more precisely, the moments of its values at a particular point.

\begin{question}\label{moments-question} For $g \in \mathbb F_q[T]$ of degree $m$ and $u \in \mathbb C$, what is

\[ \frac{1}{ | (\mathbb F_q[T]/g)^\times | } \sum_{ \chi: (\mathbb F_q[T]/g)^{\times} \to \mathbb C^{\times}} |L (\chi,u)|^{2k} ?\]

\end{question}

This question is considered most interesting when $u= |q^{-1/2}|$, because for larger, respectively smaller, values of $|u|$, the high degree, respectively low degree terms dominate $L(\chi,u)$, making the problem simpler.

This question, and its variants, have been heavily studied in number theory over the integers, and studied, though not as heavily, in the setting of $\mathbb F_q[T]$. Precise formulas predicting the value of this moment in the limit as $m$ goes to $\infty$ have been conjectured \cite{CFKRS,AK}, but they are known to hold only for $k=1$ and $k=2$ (and in the $k=2$ case, the proven bounds on the error term are much worse than the expected bounds, so more progress is possible there).

We can convert this moment problem to a counting problem by using the definition of $L(\chi,u)$ to write

\[ \sum_{ \chi: (\mathbb F_q[T]/g)^\times \to \mathbb C^\times} |L (\chi,u)|^{2k}\] \[  = \sum_{ \chi: (\mathbb F_q[T]/g)^\times \to \mathbb C^\times}  \Bigl( \sum_{d=0}^{\infty} \Bigl( \sum_{ \substack { f \in \mathbb F_q[T] \\ \textrm{monic} \\ \textrm{degree }d \\ \gcd(f,g)=1}} \chi(f) \Bigr) u^d \Bigr)^k \Bigl( \sum_{d=0}^{\infty} \Bigl( \sum_{ \substack { f \in \mathbb F_q[T] \\ \textrm{monic} \\ \textrm{degree }d \\ \gcd(f,g)=1}} \overline{\chi(f)} \Bigr) \overline{u}^d \Bigr)^k \] \[ = \sum_{ \chi: (\mathbb F_q[T]/g)^\times \to \mathbb C^\times} \sum_{d_1,\dots, d_{2k} =0}^{\infty}  \sum_{ \substack { f_1,\dots,f_{2k}  \in \mathbb F_q[T] \\ \textrm{monic} \\ \deg(f_i)= d_i \\ \gcd(f_i,g)=1}} \chi(f_1) \dots \chi(f_k)  \overline{\chi(f_{k+1})} \dots \overline{\chi(f_{2k})} u^{ d_1+ \dots + d_k} \overline{u}^{d_{k+1} + \dots + d_{2k} }\]
\[ =\sum_{d_1,\dots, d_{2k} =0}^{\infty} \Bigl(  \sum_{ \substack { f_1,\dots,f_{2k}  \in \mathbb F_q[T] \\ \textrm{monic} \\ \deg(f_i)= d_i \\ \gcd(f_i,g)=1}}  \sum_{ \chi: (\mathbb F_q[T]/g)^\times \to \mathbb C^\times} \chi(f_1) \dots \chi(f_k)  \overline{\chi(f_{k+1})} \dots \overline{\chi(f_{2k})} \Bigr) u^{ d_1+ \dots + d_k} \overline{u}^{d_{k+1} + \dots + d_{2k} } .\]
(Technically, the sum over $d_1,\dots, d_{2k}$ is divergent when $\chi$ is the trivial character because $L(1,u)$ has a pole, and so one should subtract that term off before exchanging the order of summation. We will ignore this issue to simplify our formulas and focus on the most important steps. We could instead restrict the sum to primitive odd characters using a suitable inclusion-exclusion, but the main difficulty would still be the same.)

Using orthogonality of characters, one deduces that \[ \sum_{ \chi: (\mathbb F_q[T]/g)^\times \to \mathbb C^\times} \chi(f_1) \dots \chi(f_k)  \overline{\chi(f_{k+1})} \dots \overline{\chi(f_{2k})} = \begin{cases} | (\mathbb F_q[T]/g)^\times| &   \textrm{if } f_1\dots f_k \equiv f_{k+1} \dots f_{2k}  \mod g \\ 0 & \textrm{otherwise} \end{cases} \] so we conclude that

\[ \frac{1}{ | (\mathbb F_q[T]/g)^\times | } \sum_{ \chi: (\mathbb F_q[T]/g)^\times \to \mathbb C^\times} |L (\chi,u)|^{2k} \]\[ = \sum_{d_1,\dots, d_{2k} =0}^{\infty} \Bigl(  \sum_{ \substack { f_1,\dots,f_{2k}  \in \mathbb F_q[T] \\ \textrm{monic} \\ \deg(f_i)= d_i \\ \gcd(f_i,g)=1 \\ \prod_{i=1}^k f_i\equiv \prod_{i=k+1}^{2k} f_i \mod g }} 1 \Bigr)u^{ d_1+ \dots + d_k} \overline{u}^{d_{k+1} + \dots + d_{2k} } .\]

To calculate this, it suffices to estimate
\begin{equation}\label{counting-to-estimate} \sum_{ \substack { f_1,\dots,f_{2k}  \in \mathbb F_q[T] \\ \textrm{monic} \\ \deg(f_i)= d_i \\ \gcd(f_i,g)=1 \\ \prod_{i=1}^k f_i\equiv \prod_{i=k+1}^{2k} f_i \mod g }} 1  .\end{equation}

In fact we need only estimate this for $0 \leq d_1,\dots d_{2k} \leq \deg g-1$ as if $d_i \geq \deg g$ there will be exactly $q^{d_i - \deg g}$ possible values of $f_i$ in each residue class mod $g$, which will simplify \eqref{counting-to-estimate} to an explicit formula which we can sum as a geometric series. (Or we can view these terms as the contribution of the trivial character $\chi$ and subtract them off.)   

The estimate we seek for \eqref{counting-to-estimate} will consist of a main term given as an explicit function of $d_1,\dots, d_{2k}$  and an error term which is unknown but bounded by an explicit function in $d_1,\dots, d_{2k}$. We then sum both terms over $d_i$ from $0$ to $\deg g-1$ to get an explicit estimate for the moment of $L(\chi, u)$. (This sum, because its length $(\deg g)^{2k}$ is relatively short, is much easier than the sum defining \eqref{counting-to-estimate}, which we expect to provide the main difficulty.)

The sum \eqref{counting-to-estimate} we have focused on is simply the counting problem
\[ \Bigl | \Bigl \{ f_1,\dots, f_{2k} \in \mathbb F_q[T] \mid f_i \textrm{ monic}, \deg(f_i)=d_i, \gcd(f_i,g)=1,  \prod_{i=1}^k f_i\equiv \prod_{i=k+1}^{2k} f_i \mod g \Bigr \} \Bigr | \] which by definition is \[ \left| X_{(d_i)_{i=1}^{2k}, g} (\mathbb F_q) \right| .\]

Next, we can express $X_{(d_i)_{i=1}^{2k}, g} (\mathbb F_q)$ as the solution to a system of equations over $\mathbb F_q$, minus the solutions of another system of polynomial equations. 

First, we describe the set of tuples of monic polynomials $f_i$ over $\mathbb F_q$, with $f_i$ of degree $d_i$. To do this, we need  $d_i$ variables for the coefficients of each $f_i$ (other than the leading coefficient, which is fixed at $1$), for a total of $\sum_{i=1}^k d_i$ variables.

Next, we express the condition that $\prod_{i=1}^k f_i\equiv \prod_{i=k+1}^{2k} f_i \mod g $ as a system of polynomial equations. The first step is to observe, that the coefficients of $\prod_{i=1}^k f_i$ and $\prod_{i=k+1}^{2k} f_i$ are each polynomial functions of degree $\leq k$ in the coefficients of the $f_i$, which we have chosen to be our variables. To check if $\prod_{i=1}^k f_i$ and $\prod_{i=k+1}^{2k} f_i$ are congruent mod $g$, we divide each of them by $g$ and then take the remainder, using polynomial long division. Examining the polynomial long division algorithm, one can see that the coefficients of these remainders are respectively linear functions in the coefficients of $\prod_{i=1}^k f_i$ and $\prod_{i=k+1}^{2k} f_i$, and thus are polynomial functions in the coefficients of the $f_i$. Equating the two remainders is equivalent to equating each of their $m$ coefficients, and thus it defines $m$ polynomial equations in the coefficients of the $f_i$.

Finally, we express the condition  $\gcd (f_i,g)=1$ as the negation of another system of polynomial equations. To do this, note that it is equivalent to asking that, for each prime polynomial $\pi$ dividing $g$, $f_i \not \equiv 0 \mod \pi$. Again using polynomial long division, we can express $f_i \equiv 0 \mod \pi$ as a system of equations in the coefficients of $f_i$. (Another approach is to use polynomial resultants.) 

Thus, we have described $X_{(d_i)_{i=1}^{2k}, g} (\mathbb F_q)$ as the $\mathbb F_q$-points of an algebraic variety $X_{(d_i)_{i=1}^{2k}, g}$. Therefore,  general principles (i.e. the Grothendieck-Lefschetz fixed point formula) allow us to reduce the problem of counting $X_{(d_i)_{i=1}^{2k}, g} (\mathbb F_q)$ to calculating the cohomology $H^i_c ( X_{(d_i)_{i=1}^{2k}, g, \overline{\mathbb F}_q}, \mathbb Q_\ell )$. This motivates our study of the analogous cohomology problem over the complex numbers.

\subsection{The space over $\mathbb C$} \label{ss-lf-C}

Let us now assume that the base field $K$ is the complex numbers.

In this case, our description of $X_{ (d_i)_{i=1}^{2k}, g }$ can be simplified. Because $g$ is a polynomial of degree $m$, it has exactly $m$ roots, counted with multiplicity. Let $\alpha_1,\dots, \alpha_m$ be these roots. Assume for simplicity that these roots are distinct. A polynomial is a multiple of $g$ if and only if its value at each of these roots vanishes, so two polynomials are congruent mod $g$ if and only if they have the same value at each of these roots. Thus we can re-write the definition of  $X_{ (d_i)_{i=1}^{2k}, g }$ as follows:

Let $X_{(d_i)_{i=1}^{2k}, (\alpha_i)_{i=1}^m}(\mathbb C)$ be the space of tuples $f_1,\dots, f_{2k}$ of monic polynomials with coefficients in $\mathbb C$, such that $\deg(f_i) =d_i$,  $f_i(\alpha_j) \neq 0$ for all $i$ from $1$ to $2k$ and $j$ from $1$ to $m$, and $\prod_{i=1}^k f_i(\alpha_j) =\prod_{i=k+1}^{2k} f_i (\alpha_j)$ for all $j$ from $1$ to $m$.    

In this case, we have $X_{ (d_i)_{i=1}^{2k}, g } (\mathbb C) = X_{(d_i)_{i=1}^{2k}, (\alpha_i)_{i=1}^m}(\mathbb C)$.

We want to study the cohomology with compact supports $H^*_c( X_{(d_i)_{i=1}^{2k} , (\alpha_i)_{i=1}^m}, \mathbb Q)$.

Let us see why this space satisfies the properties of Desiderata \ref{intro-vc-desiderata}.
\begin{enumerate}

\item $X_{(d_i)_{i=1}^{2k} , (\alpha_i)_{i=1}^m}$ can admit singularities. By definition, the singularities are the points where the $m \times \sum_{i=1}^{2k} d_i$ Jacobian matrix of the system of polynomials $\prod_{i=1}^k f_i(\alpha_j) - \prod_{i=k+1}^{2k} f_i (\alpha_j)$, for $j=1$ to $m$, does not have the maximum possible rank $\min(m, \sum_{i=1}^{2k} d_i)$. (This definition is justified by the fact that these will be the points where the implicit function theorem does not force $X_{(d_i)_{i=1}^{2k} , (\alpha_i)_{i=1}^m}$  to be a manifold.)

With this definition in hand, finding the singularities is an elementary problem of polynomial algebra. We will see later in Lemma \ref{singularity-characterization} that as soon as $\sum_{i=1}^{2k} d_i \geq m$, a point $(f_1,\dots, f_{2k} ) \in X_{(d_i)_{i=1}^{2k} , (\alpha_i)_{i=1}^m}(\mathbb C)$ is singular if and only if there exists a polynomial $h$, of degree $<m$, such that $h$ is a polynomial multiple of $f_i$ for all $i$.

An interesting point about these singularities is that they do not seem to appear in the number-theoretic analysis of this problem - none of the analytic tools that are used to attack the moments of $L$-functions suggest that the case where the variables all divide some number which is not too large is particularly special. They singularities only become visible when we first transfer to the function field setting and then interpret geometrically.

\item We see from our explicit description that $X_{(d_i)_{i=1}^{2k} , (\alpha_i)_{i=1}^m}(\mathbb C)$ is the solution set of $m$ equations of degree $k$ in $\sum_{i=1}^{2k} d_i$ variables, minus the solutions of $2km$ linear equations. These numbers are all relatively small - for some interesting spaces, our degrees or quantities of equations grow exponentially in terms of our parameters of interest, and this is much better. Moreover, the individual equations can be expressed nicely in terms of polynomials, which helps us do things like give a precise description of the singular locus. 

\item  While the space of all monic polynomials  $f_i$ of degree $d_i$ with $f_i(\alpha_j)\neq 0$ for $j$ from $1$ to $m$ has a nice description as a configuration space of $d_i$ points, which may collide, on $\mathbb C- \{\alpha_1,\dots , \alpha_m\}$, the condition that $\prod_{i=1}^k f_i(\alpha_j) =\prod_{i=k+1}^{2k} f_i (\alpha_j)$  does not seem to have any meaning in terms of configuration spaces, and thus there is no known purely topological description.

\item We can deform $X_{(d_i)_{i=1}^{2k} , (\alpha_i)_{i=1}^m}$ to a family of similar, but distinct spaces by varying the defining equations slightly.  More precisely, we can deform the equations $\prod_{i=1}^k f_i(\alpha_j) =\prod_{i=k+1}^{2k} f_i (\alpha_j)$ into equations $\prod_{i=1}^k f_i(\alpha_j) =c_j \prod_{i=k+1}^{2k} f_i (\alpha_j)$ for $c_j \in \mathbb C^\times$. Doing this has a dramatic effect on the singularities. In fact, one can check that for generic choices of $c$, there are no singularities at all, as the dimension of the space of possible $h$ in Lemma \ref{singularity-characterization} is less than the dimension of the space of possible tuples $c_j$. 

We can also vary the roots $\alpha_1,\dots, \alpha_m$.  However, varying the roots does not change the singularities much, and so varying $c$ seems most useful for vanishing cycles theory.

\end{enumerate}

Let us now check our characterization of these singularities, which is essentially due to Hast and Matei \cite{HastMatei}.

\begin{lemma}\label{singularity-characterization}Assume that $\sum_{i=1}^{2k} d_i\geq m$. Let $(f_1,\dots, f_{2k})$ be a point of $X_{(d_i)_{i=1}^{2k} , (\alpha_i)_{i=1}^m}$.  Then $(f_1,\dots,f_{2k})$ is a singular point of $X_{(d_i)_{i=1}^{2k} , (\alpha_i)_{i=1}^m}$ if and only if there exists an auxiliary polynomial $h$, of degree $<m$, which the $f_i$ all divide. \end{lemma}

\begin{proof} By definition, $(f_1,\dots,f_{2k})$ is a singular point if and only if the Jacobian matrix of the system of polynomials $ \prod_{i=1}^k f_i (\alpha_j) - \prod_{i=k+1}^{2k} f_i(\alpha_j)$ for $j$ from $1$ to $m$ has rank less than its maximum value, which is $\min (m, \sum_{i=1}^{2k} d_i)= m$. 

Let us first check that if such an $h$ exists, then $f_1,\dots, f_{2k}$ is singular. To do this, we use the fact that the image of the Jacobian matrix consists of the derivatives of this tuple of $m$ polynomials along every path in $\mathbb C^{ \sum_{i=1}^{2k} d_i}$ starting at $(f_1,\dots, f_{2k})$. We can view this path as as a one-parameter family of polynomials $f_i (\alpha_j,t)$ in an additional variable $t$, where $f_i(\alpha,0)=f_i(\alpha)$. Then the derivative along this path is
\[ \frac{d}{dt} \left( \prod_{i=1}^k f_i(\alpha_j, t) - \prod_{i=k+1}^{2k} f_i (\alpha_j,t)\right)\]
\[  =\left(  \prod_{i=1}^k f_i(\alpha_j, t)  \right) \sum_{i=1}^k   \frac{ \frac{d}{dt} f_i(\alpha_j, t)} { f_i(\alpha_j, t)}  - \left(  \prod_{i=k+1}^{2k} f_i(\alpha_j, t)  \right) \sum_{i=k+1}^{2k}  \frac{ \frac{d}{dt} f_i(\alpha_j, t)} { f_i(\alpha_j, t)} \]

We are interested in taking the derivative at a point which lies in $X_{(d_i)_{i=1}^{2k} , (\alpha_i)_{i=1}^m}$ so we may assume that the equations are actually satisfied at $t=0$, which causes the derivative at $t=0$ to simplify to
 
\begin{equation}\label{derivative-formula} \left(  \prod_{i=1}^k f_i(\alpha_j)  \right)   \left(  \sum_{i=1}^k   \frac{ \frac{d}{dt} f_i(\alpha_j, t)} { f_i(\alpha_j)} -  \sum_{i=k+1}^{2k} \frac{ \frac{d}{dt} f_i(\alpha_j, t)} { f_i(\alpha_j)} \right) .\end{equation}
Now each term $\frac{ \frac{d}{dt} f_i(\alpha, t) }{ f_i(\alpha,t)}$ is a rational function in $\alpha$ that vanishes at $\infty$ (because the leading coefficient of each $f_i$ is fixed at $1$) and whose denominator divides $h$ (by assumption). Hence \[ \sum_{i=1}^k   \frac{ \frac{d}{dt} f_i(\alpha, t)} { f_i(\alpha)} -  \sum_{i=k+1}^{2k} \frac{ \frac{d}{dt} f_i(\alpha, t)} { f_i(\alpha)} \] is a rational function in $\alpha$ that vanishes at $\infty$ and whose denominator divides $h$. Because this rational function vanishes at $\infty$, the degree of its numerator is strictly less than the degree of its denominator, which is at most $\deg h$ because it divides $h$. So the numerator of this sum is a polynomial degree $< \deg h  $. There are only $\deg h $ linearly independent polynomials of degree $< \deg h$. Thus for different one-parameter families $f_i(t)$ with the same $f_i(0)$, we can obtain at most $\deg h$ linearly independent tuples \[ \left( \frac{d}{dt} \left( \prod_{i=1}^k f_i(\alpha_j, t) - \prod_{i=k+1}^{2k} f_i (\alpha_j,t)\right)\right)_{j=1}^m .\]  Thus, the image of the Jacobian matrix has dimension at most $\deg h$, and it does not have full rank.

Conversely, suppose the Jacobian does not have full rank. Then there is a linear combination of the derivatives \eqref{derivative-formula} which vanishes for all families $f_i(\alpha_j,t)$. Thus there exist constants $c_j$ such that 
\[  \sum_{j=1}^m c_j  \left(  \sum_{i=1}^k   \frac{ \frac{d}{dt} f_i(\alpha_j, t)} { f_i(\alpha_j)} -  \sum_{i=k+1}^{2k} \frac{ \frac{d}{dt} f_i(\alpha_j, t)} { f_i(\alpha_j)} \right) =0 \] By Lagrange interpolation, a polynomial $\tilde{f}$ of degree $\leq m-1$ is uniquely determined by its values on $\alpha_1,\dots, \alpha_m$, so there exist constants $\lambda_j$ such that the coefficient of $T^{m-1} $ in $\tilde{f}$ is $\sum_{j=1}^{m} \lambda_j \tilde{f}(\alpha_j) $. Furthermore $\lambda_j\neq 0$ because there exists a degree $m-1$ polynomial vanishing on $\alpha_1,\dots, \alpha_{j-1}, \alpha_{j+1},\dots , \alpha_m$. Then for each $i$ we can find, by Lagrange interpolation again, a polynomial $e_i$ of degree $\leq m-1$ such that $ e_i (\alpha_j) = c_j/(f_i(\alpha_j )\lambda_j)$. 

Now $\frac{d}{dt} f_i(\alpha_j, t)$ can be an arbitrary polynomial of degree $< d_i$. It follows that for all polynomials $\tilde{f}$ of degree $<d_i$, 
\[ 0 =  \sum_{j=1}^m c_j    \frac{ \tilde{f} (\alpha_j) }{ f_i(\alpha_j)} =\sum_{j=1}^m \lambda_j   \tilde{f}(\alpha_j) e_i (\alpha_j) .\]
It follows that if $ \deg (\tilde{f}) + \deg (e_i)  = \deg (\tilde{f}e_i)\leq m-1$, then the coefficient of $T^{m-1}$  in $\tilde{f}e_i$ vanishes, so $\deg(\tilde{f}) + \deg(e_i) = \deg(\tilde{f}) + \deg(e_i) \leq m-2$. Applying this for $\tilde{f} =1$, we show that $\deg(e_i)\leq m-2$, so we can apply it for $\tilde{f} =T$, showing $\deg (e_i) \leq m-3$, and so on, until we apply it for $  T^{d_i-1}$, and show $\deg(e_i) \leq m-1-d_i$. 

Now $f_i e_i$ has degree $\leq m-1$, so is determined by its values at $\alpha_j$, which are $c_j/\lambda_j$ and in particular are independent of $i$. Hence $f_i e_i$ is equal to a fixed polynomial $h$ of degree $m-1$ for all $i$. Thus all the $f_i$ divide $h$, as desired.
%
%
%
%
%

\end{proof}

However, we do not know how to apply the vanishing cycles method, or other similar methods, to $X_{(d_i)_{i=1}^{2k} , (\alpha_i)_{i=1}^m}$ for all values of $d_1,\dots, d_{2k}$. The reason is that the singularities are too big. For example:

\begin{lemma} Assume that $d_i = d_{k+i}$ for all $i$ from $1$ to $k$. Then the dimension of the singular locus of $X_{(d_i)_{i=1}^{2k} , (\alpha_i)_{i=1}^m}$ is at least $\min(d_1 + \dots + d_k, m-1)$. \end{lemma} 

\begin{proof} For any monic polynomial $h$ of degree $m-1$, we can choose $f_i$ to be an arbitrary divisor of this polynomial with degree $d_i$ and $f_{k+i}$ to equal $f_i$, so that the equation $\prod_{i=1}^k f_i(\alpha_j) =\prod_{i=k+1}^{2k} f_i (\alpha_j)$ is automatically satisfied. This point will be singular by Lemma \ref{singularity-characterization}.  If $d_1 + \dots + d_k \geq m-1$, then we can choose $f_i$ to cover all the roots of $h$. Having done this, we will obtain distinct tuples of polynomials $f_1,\dots, f_{2k}$ for distinct $h$, so the dimension of the singular locus is at least the dimension $m-1$ of the space of possible $h$. 

If $d_1 + \dots + d_k < m-1$, we can do the same thing with $h$ of degree $d_1 + \dots + d_k $. \end{proof} 

This dimension is large enough that we don't obtain any estimates for the moments stronger than what can be obtained directly from the much easier Theorem \ref{weil}, which implies $|L(\chi,u)| \leq (1 + |u| \sqrt{q})^{m-1}$ and so the average of $|L(\chi,u)|^{2k}$ is at most $ (1+ |u|\sqrt{q})^{ 2k(m-1)}$.

However when $d_1,\dots,d_k$ is large and $d_{k+1},\dots, d_{2k}$ are small, or vice versa, the problem disappears. In fact, in the extreme case, the singularities are isolated.

\begin{lemma}\label{singular-finitely-many} Suppose $d_{k+1} = \dots = d_{2k}= 0 $. Then the singular locus of $X_{(d_i)_{i=1}^{2k} , (\alpha_i)_{i=1}^m}$ consists of finitely many points. \end{lemma}
\begin{proof} Let $f_1,\dots,f_{2k}$ be a singular point. We apply Lemma \ref{singularity-characterization} to conclude that there exists an $h$ which all the $f_i$ divide. Let $\beta_1,\dots, \beta_{\deg h}$ be the roots of $h$.

Any monic $f_i$ dividing $h$ must have the form \[ f_i(T)  =\prod_{j'=1}^{\deg h} (T- \beta_{j'})^{e_{i,j'}}\]  for some $e_{i,j'} \in \{0,1\}$, so \[ \prod_{i=1}^k f_i(T) = \prod_{j'=1}^{\deg h} (T- \beta_{j'})^{e_{j'}}\]  for some $e_{j'} \in \mathbb N$. By removing the root $\beta_j$ from $h$ in case $e_j'=0$, we may assume that all the $e_j'$ are positive. There are finitely many possible values for the $e_{i,j'}$, and for each of them, the system of equations
\[  \prod_{j'=1}^{\deg h} (\alpha_1- \beta_{j'})^{e_{j'}} = 1 , \dots,  \prod_{j'=1}^{\deg h} (\alpha_m- \beta_{j'})^{e_{j'}} = 1 \] has finitely many solutions. 

To check this, it suffices to show that the Jacobian of this system of $m$ equations in $\deg h$ variables has full rank. We calculate the Jacobian using logarithmic derivatives.  This gives a matrix $M$ whose entries are $M_{j j'} =  \frac{ -  e_{j'} } { \alpha_j - \beta_{j'}  }$ which cannot have a kernel as, for a vector $(c_{j'})$ in the kernel, $\sum_{j'=1}^{\deg h} \frac{ - e_{j'} c_{j'}}{ T -\beta_{j'}}$ would be a nonzero rational function with $\deg h$ poles, each of multiplicity one, and $m$ zeroes, vanishing at $\infty$, which is impossible as $\deg h<m$.\end{proof}

Because the singularities are isolated, applying the vanishing cycles method to control the singularities of $X_{(d_i)_{i=1}^{2k}, (\alpha_i)_{i=1}^{2k}}$ in the case $d_{k+1}=\dots =d_{2k}=0$ looks promising. However, we are far from done with the problem at this point. To apply the method, we also need to compactify the variety, understand the singularities and vanishing cycles at the points at infinity of this compactification,  bound the dimension of the vanishing cycles sheaf at each isolated singularity, and calculate the cohomology of the generic fiber.

All of these steps were carried out for a slight variant of this problem in \cite{square-root}. In this variant, rather than demanding $\prod_{i=1}^k f_i \equiv 1$ mod $g$, we demand the coefficients of $T^{d_1+ \dots + d_k-1},\dots T^{d_1+\dots + d_k-m}$ in $\prod_{i=1}^k f_i$ vanish. This variant is almost the same as the case when $g= T^m$, using the change of variables $T \mapsto T^{-1}$ that reverses the order of the coefficients of a polynomial, but because monicity is not preserved by this change of variables it ends up relating to the moments of only those $\chi$ modulo $g$ with $\chi(\mathbb F_q^\times)=1$ (i.e. the even characters). 

 This vanishing cycles argument in \cite{square-root} concludes in showing that $H^i_c( X_{ (d_i)_{i=1}^{2k} }, \mathbb Q)=0$ for all \[ i \in ( d+1, 2 d) \] where \[ d= \left(\sum_{i=1}^{k} d_i \right) - m = \dim X_{ (d_i)_{i=1}^{2k} },\] that $H^{2d}_c( X_{ (d_i)_{i=1}^{2k} }, \mathbb Q)$ is one-dimensional, and that the dimensions of the cohomology groups in degrees $\leq d+1$ are bounded by $3 (2k+2 )^{ 2m+d }$. Similar results, with a slightly weaker cohomology vanishing statement, were proved for $\ell$-adic cohomology in characteristic $p$.

Obtaining these results requires studying a compactification. A good choice involves embedding $X_{ (d_i)_{i=1}^{2k} }$ into the weighted projective space where the $j$th coefficient of a polynomial $f_i$ is a coordinate of degree $d_i - j$. This makes the coefficient of $T^{d_1+\dots + d_k-r}$ homogeneous of degree $k$, so the equations defining $X_{ (d_i)_{i=1}^{2k}}$ are each homogeneous. (Alternatively, we can reduce to the case where each $d_i=1$ by factoring each $f_i$ into linear factors. In this case, our compactification embeds into an unweighted projective space.) A variant of Lemma \ref{singular-finitely-many} works in this setting, showing that the singularities of the compactification are isolated. (In characteristic $p$, we can only give an upper bound on the dimension of the singular locus, which improves as $p$ grows). We can apply vanishing cycles in a family where we vary the defining equations of $X_{(d_i)_{i=1}^{2k}}$ to the equations defining a smooth affine complete intersection. The analogous vanishing result for the cohomology of the affine complete intersection is well-known. Because the singularities are isolated, the vanishing cycles sheaf is supported at finitely many points, so its cohomology satisfies a similar vanishing result. By applying the vanishing cycles long exact sequence, we can conclude that $H^i_c( X_{ (d_i)_{i=1}^{2k} }, \mathbb Q)=0$ for $d+1< i < 2d$ and is one-dimensional if $i=2d$.

Finally, the bounds for the dimension of the cohomology groups use a separate method of Katz which proves bounds for the cohomology groups of a variety in terms of the number of variables, number of defining equations, and degrees of defining equations \cite{katzbetti}.

By proving these results on cohomology, \cite{square-root} was able to obtain estimates for sums of divisors functions, as well as a variant of Question \ref{moments-question} where the $L$-function is not  encased in an absolute value, i.e. an estimate for  \[ \frac{1}{ | (\mathbb F_q[T]/g)^\times | } \sum_{ \chi: (\mathbb F_q[T]/g)^{\times} \to \mathbb C^{\times}} L (\chi,u) ^{k} .\] To see why this estimate was obtained, note that following the arguments of \S\ref{ss-lf-bm} without the absolute value, we obtain the counting function in the special case where $d_{k+1}= \dots = d_{2k}=0$. 

It is likely possible to carry out a similar argument for values of $g$ other than $T^m$. In the case when $g$ has distinct roots, this will be done in \cite{square-free}. Here a suitable compactification embeds into a product of projective spaces, where the coefficients of each $f_i$ are the coordinates of a different projective space.  We can apply vanishing cycles in the family varying $c_j$ discussed above. In this case, the compactification adds new singularities, but their structure is so simple that we can explicitly check there are no vanishing cycles at these singular points. Here the cohomology of the generic fiber is much subtler to compute and requires a careful analysis of the compactification and the theory of perverse sheaves.  The final result is a calculation of the compactly supported cohomology of $X_{ (d_i)_{i=1}^{2k} ,(\alpha_i)_{i=1}^{m}}$ in all degrees but $\dim X_{ (d_i)_{i=1}^{2k} ,(\alpha_i)_{i=1}^{m}}$ and $\dim X_{ (d_i)_{i=1}^{2k} ,(\alpha_i)_{i=1}^{m}} +1$, and bounds on the dimensions of cohomology in those degrees (still in the case $d_{k+1}=\dots = d_{2k}=0$.) This again leads to estimates for moments without the absolute value, this time for characters $\chi$ modulo a squarefree polynomial $g \in \mathbb F_q[T]$.

\subsection{Future work} 

It would be of great interest to generalize the vanishing cycles techniques to general polynomials $g$ which are neither squarefree not a power of $T$, or to handle the case when $d_i$ is large for both $1\leq i \leq k$ and $k+1 \leq i \leq 2k$, even for very specific  $g$.  For general $g$, the main difficulty is finding a good compactification, but to handle the case where all $d_i$ are large, more ideas than that are necessary.

\section{Short character sums}\label{s-cs}

Many problems in analytic number theory involve, at a crucial point, short character sums. Let $M$ be a natural number and let $\chi: (\mathbb Z/M)^{\times} \to \mathbb C^\times$ be a homomorphism.  Extend $\chi$ to a function on $\mathbb Z$ by setting $\chi(n)=0$ if $\gcd(n,m) \neq 1$ (and therefore $n$ is not an invertible element of $(\mathbb Z/M \mathbb Z)^{\times}$). Then a short sum of $\chi$ has the form   \begin{equation}\label{Z-character-sum} \sum_{n=a}^{a+N-1}  \chi(n) \end{equation} for $a \in \mathbb Z$ and $N < M$.  Because $(\mathbb Z/M \mathbb Z)^{\times}$ is a finite group, $\chi$ takes values in complex numbers of modulus $1$. Because each term has modulus at most $1$, the size of this sum is at most the total number of terms, which is $N$. Thus the \emph{trivial bound} for this sum is \[ \left|  \sum_{n=a}^{a+N-1}  \chi(n)\right| \leq N. \]   
If $\chi$ is the trivial character, this bound is essentially sharp, but for any other $\chi$, the trivial bound can often be improved. The most important improvements are the P\'olya-Vinogradov inequality \cite{Polya,Vinogradov} \[ \left|  \sum_{n=a}^{a+N-1}  \chi(n)\right| \leq \sqrt{M} \log M ,\]  which improves on the trivial bound for $N > \sqrt{M} \log M $, and the Burgess bound \cite{Burgess}, which is more complicated to state, but  improves on the trivial bound whenever $N> M^{1/4+ \epsilon}$ for any fixed $\epsilon>0$. Bounds which improve on the trivial bound for a greater range of possible $N$, or bounds which improve on it to a greater extent, would have numerous applications in analytic number theory. 

It turns out that this problem has a very natural topological analogue, at least in the case when the modulus $M$ is squarefree, and this analogue can be almost completely solved. However, the translation between arithmetic and topology is nontrivial. Thus, we will explain the topological problem first, then explain the relationship with character sums, and finally explain what makes it solvable.

Let $n$ and $m$ be natural numbers. Let $\alpha_1,\dots, \alpha_m \in \mathbb C$ be a tuple of distinct complex numbers, and let $\beta_1,\dots, \beta_m$ be another tuple of complex numbers, not necessarily distinct. Let $X_{(\alpha_i)_{i=1}^m , (\beta_i)_{i=1}^m}$ be the moduli space of polynomials $h$ of degree $<n$ such that $h(\alpha_i) \neq \beta_i$ for $i$ from $1$ to $m$.

Note that the space of polynomials $h$ of degree $<n$ is simply $\mathbb C^n$, as we have one coordinate for each coefficient of the polynomial. Using these coordinates, $h(\alpha_i)$ is a linear function on $\mathbb C^n$, so the set of $h$ such that $h(\alpha_i) = \beta_i$ is a hyperplane in $\mathbb C^n$, and thus $X_{(\alpha_i)_{i=1}^m , (\beta_i)_{i=1}^m}$ is the complement of $m$ hyperplanes in $\mathbb C^n$. So our space under consideration is a special type of hyperplane complement.

\begin{question} What is \[ H^* (X_{(\alpha_i)_{i=1}^m , (\beta_i)_{i=1}^m}, V) \] where $V$ is a nontrivial one-dimensional representation of $\pi_1 ( X_{(\alpha_i)_{i=1}^m , (\beta_i)_{i=1}^m})$?  \end{question}

In other words, we consider, not the usual cohomology, but cohomology twisted by some local system of rank $1$ on our space. Other than that, the situation is a standard one for arithmetic topology - our goal is to show vanishing of the high-degree compactly supported cohomology groups, or, equivalently (because Poincar\'{e} duality is valid in this setting), of the low-degree usual cohomology groups, as well as some reasonable bound for the dimensions of the middle groups.

\subsection{The relation between topology and arithmetic}

The first step in the translation to topology is to develop $\mathbb F_q[T]$-analogues of every element of our arithmetic problem. We can replace $M$ with a monic polynomial $g \in \mathbb F_q[T]$. Let $m$ be the degree of $g$.

We can then define $\chi$ as a homomorphism $( \mathbb F_q[T] / g  )^\times \to \mathbb C^\times$, and extend it by zero to non-invertible elements.

To find the appropriate analogue of an interval, we observe that the set of all $h \in \mathbb F_q[T]$ with $\deg h < n$  behaves like an interval around $0$, in particular because it contains all polynomials whose norm $|h| = q^{ \deg h}$ is small, and thus the set of polynomials $f+h$, where $f$ is fixed and $h$ varies over polynomials with $\deg h < n$, behaves like an interval around $f$. Thus, the function field analogue of \eqref{Z-character-sum} is \begin{equation}\label{q-character-sum} \sum_{ \substack{ h\in \mathbb F_q[T] \\ \deg h < n}} \chi(f+h) .\end{equation}   However, it is equivalent, and will be more convenient later, remove the terms where $\chi$ is set to $0$ from the sum, obtaining \[ \sum_{ \substack{ h\in \mathbb F_q[T] \\ \deg h < n \\ \gcd(f+h, g) = 1 }} \chi(f+h) .\] This will be convenient as it puts more of the complexity into the set we are summing over rather than the function being summed, and the set is easier to interpret geometrically. 

For the trivial bound, we can observe that there are $q^n$ possible $h$ with $\deg h<n$, as each coefficient can take $q$ values, so we have \[ \Bigl|  \sum_{ \substack{ h\in \mathbb F_q[T] \\ \deg h < n \\ \gcd(f+h, g) = 1 }} \chi(f+h)  \Bigr| \leq q^n.\] Again, the key problem is to substantially improve the trivial bound when $\chi$ is not the constant function $1$.

To do this, we let $X_{g,f}(K)$ for a monic polynomial $g$ and a polynomial $f$, both over a field $K$, be the set of polynomials $h$ over $K$ with $\deg h<n$ and $\gcd(f+h, g)=1$.

Over the complex numbers, we can factor $g = \prod_{i=1}^m (T-\alpha_i)$. In this case, $\gcd(f+h, g)\neq 1$ if and only if some factor $(T-\alpha_i)$ divides $f+h$, which happens exactly when $h(\alpha_i) + f(\alpha_i) =0$, or $h(\alpha_i) = -f(\alpha_i)$. Thus, setting $\beta_i = - f(\alpha_i)$, the set of polynomials $h$ with $\gcd(f+h, g) = 1$ is exactly the set where $h(\alpha_i) \neq \beta_i$ for all $i$.

Hence $X_{(\alpha_i)_{i=1}^m , (\beta_i)_{i=1}^m}(\mathbb C) = X_{g, f}(\mathbb C) $ is the topological analogue of the set $X_{g,f} (\mathbb F_q)$ we sum over in \eqref{q-character-sum}. We have thus translated every aspect of our arithmetic problem into topology, except for the character $\chi$. Why should the analogue of summing the character $\chi$ be taking cohomology with coefficients in a representation $V$ of the fundamental group? The answer comes from the twisted Grothendieck-Lefschetz formula (\cite[Rapport, Theorem 3.2]{sga4h} or \cite[Theorem 10.5.1]{Fu}).

For a space $X$ over a finite field $\mathbb F_q$, the usual Grothendieck-Lefschetz formula expresses $|X(\mathbb F_q)|$ in terms of the trace of a Frobenius element $\Frob_q$ on the compactly-supported cohomology $H^*_c( X_{\overline{\mathbb F}_q}, \mathbb Q_\ell)$. The twisted Grothendieck-Lefschetz formula does the same thing for the cohomology $H^*_c( X_{\overline{\mathbb F}_q},V)$, where $V$ is a representation of the fundamental group of $X$. Because we are working in the finite field setting, it is necessary that we use the \'{e}tale fundamental group $\pi_1^{\textrm{et}} (X)$. So let us take $V$ to be a representation of $\pi_1^{\textrm{et}} (X)$ over $\mathbb Q_\ell$.

For us, the crucial property of $\pi_1^{\textrm{et}}(X)$ is that, for any covering space $\pi: Y\to X$ (technically, a finite \'{e}tale map), and any $x \in X(\mathbb F_q)$, $\pi_1^{\textrm{et}}(X)$ acts on the $\overline{\mathbb F}_q$-points of the fiber $\pi^{-1}(x)$ of $\pi$ over $x$. Furthermore, the isomorphism class of this fiber, as a set with a $\pi_1$-action, is independent of $x$. This is the analogue of the classical monodromy action on the fiber and its invariance under change of fiber. (More carefully, we need to pick a base point for the fundamental group, and then the action of the fundamental group on the fiber over another base point is well-defined only up to conjugacy, but our actions will factor through abelian groups so conjugacy is trivial and thus this detail is irrelevant.)

It turns out there are special conjugacy classes $\Frob_{q,x} \in \pi_1^{\textrm{et}}(X)$, for each $x \in X(\mathbb F_q)$, which are characterized by the property that the action of $\Frob_{q,x}$ on a point $y\in \pi^{-1}(x)$ is the same as the action of $\Frob_q$ - in other words is the same as raising all coordinates of $y$ to their $q$th power. The twisted Grothendieck-Lefschetz fixed point formula states 
\[ \sum_{x \in X(\mathbb F_q)} \tr( \Frob_{q,x}, V)  = \sum_{i=0}^{2\dim X} (-1)^i \tr( \Frob_q,  H^i_c( X_{\overline{\mathbb F}_q},V)).\]
 
Thus, because we know that $X_{g,f, \mathbb F_q}( \mathbb F_q) $ is exactly the set of $h \in \mathbb F_q[T], \deg h< n, \gcd(f+h, g)=1$ we wish to sum over, to interpret \eqref{q-character-sum} geometrically, it suffices to find a representation $V$ of $\pi_1^{\textrm{et}} ( X_{g,f, \mathbb F_q})$ such that $\tr( \Frob_{q,h}, V)=\chi(f+h)$.

To find such a representation, we use a strategy due to Lang \cite[\S3]{Lang}.

Let $Y_{g,f}$ be the space with coordinates $a_0,\dots, a_{n-1}$, which we view as the coefficients of a polynomial $ \sum_{i=0}^{m-1} a_i T^i $, and $h_0,\dots, h_{n-1}$, which we view as coefficients of a polynomial $\sum_{i=0}^{n-1} h_i T^i$, subject to the conditions
\[ \gcd\left( \sum_{i=0}^{m-1} a_i T^i  ,g\right) =1, \gcd \left( f+ \sum_i h_i T^i , g\right)=1,\]
and
\begin{equation}\label{magic-Lang-formula} \sum_{i=0}^{m-1} a_i^q  T^i \equiv \left( \sum_{i=0}^{m-1} a_i T^i  \right) \left(f + \sum_{i=0}^{n-1}  h_i T^i \right) \mod g  .\end{equation}
In particular, the important equation is the last one, because this ``forces" $\Frob_q$ to act by multiplication by $f+h$. More precisely, this is a system of $m$ equations, one for each coefficient mod $g$.

Let $\pi: Y_{g,f} \to X_{g,f}$ be the map that forgets $a$ and remembers $h$. Lang \cite{Lang} proved the key facts:

\begin{theo}\label{thm-Lang} \begin{enumerate}

\item $\pi$ is a finite \'{e}tale map, so $Y_{g,f}$ is a covering space of $X_{g,f}$.

\item $(\mathbb F_q[T]/g)^\times$ acts on $Y_{g,f}$ by fixing $h$ and acting by multiplication modulo $g$ on the polynomial $ \sum_{i=0}^{m-1} a_i T^i$.

\item This action is simply transitive on the fiber of $\pi$ over any point of $X_{g,f}$ defined over an algebraically closed field.

\item The action of each element of $\pi_1^{\operatorname{et}} ( X_{g,f})$ on a fiber of $\pi$ is equal to the action of some unique element of  $(\mathbb F_q[T]/g)^\times$, defining a homomorphism $\pi_1^{\operatorname{et}} (X_{g,f} ) \to (\mathbb F_q[T]/g)^\times$.

\item The image of $\Frob_{q,h}$ under this homomorphism, for a point $h \in X_{g,f}(\mathbb F_q)$, is $f+h$. \end{enumerate}

\end{theo}

Indeed our (1) is \cite[discussion before Theorem 2 on p. 557]{Lang}, (2) and (3) are \cite[(3) on p. 556]{Lang}, and (4) and (5) are the $d=1$ case of \cite[(6) on p. 560]{Lang}, though all of these are expressed in different language.

\begin{proof} There are many things to check, and all can be done straightforwardly, so we just give a sketch, except for the most important ones. 

For (1), the key point is that when differentiating the equation \eqref{magic-Lang-formula} with respect to $a_i$, we can treat the terms $a_i^q$ as constant because the derivative of $a_i^q$ is a multiple of $q$ and thus vanishes.

For (2), the key point is that, because elements of $\mathbb F_q$ are preserved by raising to the $q$th power, when we multiply $\sum_i a_i T^i$ by a polynomial with coefficients in $\mathbb F_q$, we multiply on the left and right side of \eqref{magic-Lang-formula} by the same polynomial, and so \eqref{magic-Lang-formula} is preserved.

For (3), the key point is the reverse - the only elements of any field of characteristic $p$ preserved by raising to the $q$th power are those in $\mathbb F_q$. When we divide two solutions $\sum_i a_i T^i$ of \eqref{magic-Lang-formula}, with the same $h$, by each other, we get a polynomial invariant under raising each coefficient to the $q$th power, and this fact enables us to conclude that its coefficients are in $\mathbb F_q$.

(4) follows from (3) and some naturality properties of the $\pi_1$ action. This action is defined in an isomorphism-invariant way, which means that the action of any element of $\pi_1(X_{g,f})$ commutes with the action of any automorphism of $Y_{g,f}$ fixing $X_{g,f}$. But for a finite abelian group $G$ acting simply transitively on a finite set, the only permutations that commute with the action of $G$ are precisely the elements of $G$.

(5) is the crux of the matter, because it is what allows us to explain why we chose the particular equation \eqref{magic-Lang-formula}, rather than listing the miraculous properties of an equation which seemingly appeared out of nowhere. But it is also straightforward. We have defined the action of $(\mathbb F_q[T]/g)^\times$ as by multiplication on $\sum_i a_i$ modulo $g$. We have also defined $h =\sum_i h_i T^i$ with $h_i \in \mathbb F_q$. Thus, to guarantee that $\Frob_{q,h}$ raising each coordinate to the $q$th power is equivalent to multiplying by $f+h$ modulo $g$, we need \[ \sum_{i=0}^{m-1} a_i^q  T^i \equiv \Bigl( \sum_{i=0}^{m-1} a_i T^i  \Bigr) \Bigl(f + \sum_{i=0}^{n-1}  h_i T^i \Bigr) \mod g \] which is precisely \eqref{magic-Lang-formula}.

 \end{proof}

It follows from Theorem \ref{thm-Lang} that if we compose this homomorphism $\pi^{\textrm{et}}_1(X_{M, f, \mathbb F_q})\to ( \mathbb F_q[T]/ M\mathbb F_q[T])^\times$ with a one-dimensional representation $\chi$ of $( \mathbb F_q[T]/ M\mathbb F_q[T])^\times$, we obtain a one-dimensional representation $V$ of $\pi^{\textrm{et}}_1(X_{M, f, \mathbb F_q})$ such that $\tr( \Frob_{q,h}, V)=\chi(f+h)$, as desired.

The Grothendieck-Lefschetz formula now tells us how to express the sum of $\chi(f+h)$ in terms of the compactly supported cohomology of $V$. Because $X$ is smooth, we may as well consider the usual cohomology of $V^\vee$, which is also a one-dimensional representation.

This explains why the complex analogue has cohomology twisted by a one-dimensional representation.

\begin{remark}  A subtlety here is that the Grothendieck-Lefschetz formula applies over $\mathbb Q_\ell$ and we wish to study exponential sums over $\mathbb C$. To fix this we fix some isomorphism between the algebraic closure of $\mathbb Q_\ell$ and $\mathbb C$, which lets us freely transition between the two of them. To avoid using the axiom of choice, we could instead choose an isomorphism between some sufficiently large subfields of those fields. This technical detail almost never causes trouble and thus should be ignored whenever possible. \end{remark}

\subsection{Topological methods}

Let us now examine the spaces $X_{(\alpha_i)_{i=1}^m , (\beta_i)_{i=1}^m}$ from the perspective of topology. There are multiple methods to studying the twisted cohomology of hyperplane complements that could be applied to $H^* (X_{(\alpha_i)_{i=1}^m , (\beta_i)_{i=1}^m}, V)$ - in the complex setting these include \cite{Nonresonance} and related works giving criteria for nonvanishing outside the middle degree. Here we will focus specifically on the approach by vanishing cycles theory. To make this work, as a starting point one must compactify the space $X_{(\alpha_i)_{i=1}^m , (\beta_i)_{i=1}^m}$. Because $X_{(\alpha_i)_{i=1}^m , (\beta_i)_{i=1}^m}$ is a hyperplane complement in $\mathbb C^n$, a suitable compactification is $\mathbb P^n (\mathbb C)$. Having done this, the singularities we are required to study by vanishing cycles theory are the singularities of the pair $( \mathbb P^n (\mathbb C), X_{(\alpha_i)_{i=1}^m , (\beta_i)_{i=1}^m})$, or, perhaps more clearly, the divisor $\mathbb P^n(\mathbb C) \setminus X_{(\alpha_i)_{i=1}^m , (\beta_i)_{i=1}^m}$, a union of projective hyperplanes.

Let us now now see how the spaces  $X_{(\alpha_i)_{i=1}^m , (\beta_i)_{i=1}^m}$ satisfy the four properties from Desiderata \ref{intro-vc-desiderata} that make vanishing cycles theory a good fit.  

\begin{enumerate}

\item While $X_{(\alpha_i)_{i=1}^m , (\beta_i)_{i=1}^m}$ and its compactification $\mathbb P^n(\mathbb C)$ are both manifolds, the difference $\mathbb P^n(\mathbb C) \setminus X_{(\alpha_i)_{i=1}^m , (\beta_i)_{i=1}^m}$ is not. 

\item However, $\mathbb P^n(\mathbb C) \setminus X_{(\alpha_i)_{i=1}^m , (\beta_i)_{i=1}^m}$ is the solution set of a nice set of equations - linear ones, in fact.

\item $X_{(\alpha_i, (\beta_i)_{i=1}^m}$ has a nice purely topological description as a hyperplane complement, but we do not use this description.

One reason this might be a good idea is the nontrivial difficulty here in bringing a purely topological argument from characteristic zero to characteristic $p$. We have throughout this article mostly ignored the differences between these two settings. In part, this is because vanishing cycles arguments in characteristic zero and characteristic $p$ are similar to each other. For other topological methods, one usually needs a separate argument to show a solution of the characteristic zero problem implies a solution of the characteristic $p$ problem. 

It turns out that, in this case, deducing characteristic $p$ from characteristic $0$ would likely require vanishing cycles arguments  - in fact many of the same ones that can be used to directly study the cohomology in characteristic $p$. So knowledge in characteristic zero might not be so helpful on this problem.

\item A suitable family of spaces is provided by varying the parameters $\beta_i$. One could also vary the $\alpha_i$, but this is not necessary.

\end{enumerate}

For $X_{(\alpha_i)_{i=1}^m , (\beta_i)_{i=1}^m}$, as the $\beta_i$ vary, it turns out that, by \cite[XIII Lemma 2.1.11]{sga7-ii}, the vanishing cycles complex is supported at only the points where the hyperplanes $f(\alpha_i)=\beta_i$  fail to have normal crossings. Concretely, this means that only the points where some $k$ hyperplanes intersect in a linear space of dimension greater than the expected $n-k$ have vanishing cycles.

In our case, this means the vanishing cycles are supported at finitely many points, and we are again in the happy situation where the cohomology of vanishing cycles is simply a sum of contributions from special points:

\begin{lemma} Fix a tuple $(\alpha_i)_{i=1}^m , (\beta_i)_{i=1}^m$. Given $k$ indices $i_1,\dots, i_k$, if the $k$ hyperplanes of the form $f(\alpha_{i_1})=\beta_{i_1},\dots , f(\alpha_{i_k}) = \beta_{i_k}$ intersect in a linear space of dimension greater than $n-k$, then in fact $k>n$ and these hyperplanes intersect in a single point (i.e. a linear space of dimension $0$).   \end{lemma}

Since there are finitely many sets of $k$ indices from $1$ to $m$, the total number of points that can appear this way is finite. 

\begin{proof} Consider the solutions of the equations $h(\alpha_i)=\beta_i$ for $k$ values $i_1,\dots i_k$ of $i$. If polynomials $h_1$ and $h_2$ both satisfy these equations, then $h_1 (\alpha_i) - h_2(\alpha_i)=0$ for all such $i$, then $h_1-h_2$ is divisible by the polynomial $\prod_{j=1}^k (T- \alpha_{i_j})$. The space of multiples of $\prod_{j=1}^k (T- \alpha_{i_j})$ with degree $<n$  has dimension exactly $\max(n-k,0)$, being generated by \[ \prod_{j=1}^k (T- \alpha_{i_j}),\hspace{5pt} T \prod_{j=1}^k (T- \alpha_{i_j}),\hspace{5pt} T^2 \prod_{j=1}^k (T- \alpha_{i_j}),\hspace{5pt} \dots  \hspace{5pt}T^{n-k-1} \prod_{j=1}^k (T- \alpha_{i_j}).\] Thus, the dimension of the space of solutions to these $k$ equations is $\max(n-k,0)$. Hence, the dimension of the intersection of these $k$ hyperplanes is only greater than $ n-k$ if $k>n$, in which case the dimension is zero and the intersection is a point.\end{proof} 

To calculate the vanishing cycles contribution from a single polynomial $h$, we use our typical strategy of passing to a local model where the geometry is simpler. Our local model will be the complement in $\mathbb P^n (\mathbb C)$ of all the hyperplanes that contain the point $h$. In other words, the hyperplanes which stay away from our given point can be ignored. In this setting, the cohomology of the special fiber, and vanishing cycles, can be computed explicitly. They exist only in the middle degree, from which it follows that \[ H^i(X_{(\alpha_i)_{i=1}^m , (\beta_i)_{i=1}^m}, V) =0 \] unless $i= n-1$ or $n$. The computations can also used to show that the Betti numbers in those degrees are at most $\binom{m-1}{n} $. This gives a very strong control on the corresponding sum. In fact, in \cite{SS}, myself and Shusterman show that
\[  \Bigl|  \sum_{ \substack{ h\in \mathbb F_q[T] \\ \deg h < n \\ \gcd(f+h, g) = 1 }} \chi(f+h)  \Bigr|  \leq (\sqrt{q}+1) \binom{m-1}{n} q^{n/2} .\] This minimum size of $n$ in terms of $m$ where this bound beats the trivial bound $q^n$ gets larger as $q$ grows larger. For small $q$, it is weaker than the Burgess bound and only helpful for $n>m/4$, whereas for large $q$ it can be much stronger.

Furthermore in \cite{SS}, we used a relationship between the M\"{o}bius function in number theory and characters, to deduce further arithmetic results about the M\"{o}bius function and primes. This relationship is only valid for polynomials over fields of small characteristic $p$, and does not hold for integers or polynomials over the complex numbers. It relies essentially on the fact that, for a polynomial $f$ in $T$, \[\frac{d}{dT} f^p = p  f^{p-1} \frac{df}{dT}  = 0 ,\] which does not have an analogue over the complex numbers or in the integers. This leaves the topological analogues of these questions about the M\"{o}bius function completely open.

\subsection{Future work}

It would be interesting to apply these methods to more general hyperplane complements that arise in number theory, in particular those defined by $h(\alpha_i) \neq \beta_{i,j}$ with multiple forbidden values $\beta_{i,j}$ for each $\alpha$. These would be connected to sums like \[ \sum_{n=a}^{a+N-1}  \chi_1(n) \chi_2(n+c) \] for a constant $c$. 

It would also be interesting to see if the results about the M\"{o}bius function, proven in small characteristic by reduction to Dirichlet characters, can be obtained in larger characteristics by some more subtle geometric method.

 \section{Equidistribution of CM points}\label{s-st}

Let $C$ be a hyperelliptic curve of genus $g$ over a finite field or the complex numbers. In other words, $C$ is the double cover of $\mathbb P^1$ branched at $2g+2$ fixed points of $\mathbb P^1$, or more generally a degree $2g+2$ subscheme of $\mathbb P^1$. Let $J$ be the Jacobian of $C$, a $g$-dimensional abelian variety. Then $J$ carries the structure of a commutative group. Furthermore, there is an Abel-Jacobi map $C \to J$. (To define this map, we need to fix a degree $1$ divisor class on $C$, which is always possible over complex numbers or a finite field, but may not be over other fields.)  By adding together $a$ copies of the Abel-Jacobi map using the group structure, we obtain a map $C^a \to J$. Let $\Theta_a$ be its image inside $J$. For a point $x \in A$, Shende and Tsimerman \cite{ST} studied the intersection \begin{equation}\label{theta-counting}\Theta_a \cap [x + \Theta_b],\end{equation} where $[x+ \Theta_b]$ refers to translating $\Theta_b$ by $x$ using this group structure.  Specifically, they studied its cohomology over the complex numbers and its number of points over finite fields.

The relationship of these two problems to number theory is somewhat subtle. Given an imaginary quadratic field $K = \mathbb Q( \sqrt{-D})$, the ring of integers \[\mathcal O_K =\begin{cases} \mathbb Z[\sqrt{-D} ] & D\not\equiv 3 \mod 4 \\ \mathbb Z\left[ \frac{ 1+ \sqrt{-D} }{2} \right] & D\equiv 3\mod 4 \end{cases}\] embeds into the complex numbers $\mathbb C$, forming an elliptic curve $\mathbb C/ \mathcal O_K$. This defines a point in the space $X(1) = \mathbb H/ SL_2(\mathbb Z)$ parameterizing elliptic curves. Moreover, for any ideal $I$ in the ring of integers $\mathcal O_K$, $\mathbb C/I$ is another elliptic curve. Two ideals give isomorphic elliptic curves if and only if they are equal up to multiplication by an element of $K$, and the equivalence classes of ideals up to this operation from a group, the ideal class group $Cl(K)$.  These elliptic curves are known as elliptic curves with \emph{complex multiplication}, and the associated points of $X(1)$ are \emph{CM points}. Duke's theorem \cite{Duke} tells us that, for $U$ a reasonable open subset of $X(1)$
\[ \lim_{D \to \infty}  \frac{ \left| \left\{ I \in Cl(K) \mid (\mathbb C/I) \in U \right \} \right| }{   |Cl(K)|} = \mu(U) \] where $\mu$ is the $SL_2(\mathbb R)$-invariant measure on $\mathbb H$ that assigns $X(1)$ total mass one, and an open set is reasonable if its boundary has measure zero. In other words, this theorem shows that the elliptic curves $\mathbb C/I$ are uniformly distributed inside $X(1)$, according to the measure $\mu$, in the limit as $D \to \infty$.

A question of Michel and Venkatesh asks, for a sequence of discriminants $D$ and ideals $I_D$, and two reasonable open sets $U_1, U_2$, if
\[ \lim_{D \to \infty}  \frac{ \left| \left\{ I \in Cl(K) \mid (\mathbb C/I) \in U_1,  (\mathcal C/ I_DI ) \in U_2\right\} \right| }{   |Cl(K)|} = \mu(U_1) \mu(U_2) \] as long as the minimum norm of an ideal representing $I_D$ goes to $\infty$ \cite[Conjecture 2 on p. 7]{MV}. In other words we ask whether, not only is $(\mathbb C/I)$ distributed uniformly, and therefore $(\mathbb C/I_DI)$ is distributed uniformly, but in addition these two distributions are independent.

It turns out that the question studied in \cite{ST} is a function field analogue of the Michel-Venkatesh equidistribution question. After \cite{ST} was written, a modified form of the equidistribution question was answered by Khayutin \cite{Khayutin}, along with a general version involving more than two sets $U_1,U_2$. The modification was that the limit must be restricted to discriminants $D$ such that $-D \mod p_i$ is a nonzero square in $\mathbb Z/p_i$ for two fixed small primes $p_1,p_2$. This uses techniques from ergodic theory that are completely different from those of \cite{ST}.

\subsection{The analogy between topology and arithmetic}

This explanation will be the sketchiest, as the techniques used to the number theoretic and topological pictures here are the furthest afield from what is discussed in the rest of this article. Thus, it may be better to view this as an explanation of an analogy rather than a rigorous explanation of a correspondence. This can all be given a very rigorous reasoning, in terms of the adelic point of view, as is discussed in \cite{ST}.

To make the analogy work, it is helpful to think of the elliptic curves parameterized by $X(1)$ not as elliptic curves, but merely as lattices of rank two. In other words, these are rank two free $\mathbb Z$-modules with a positive definite symmetric bilinear form, defined up to scaling.  The analogue of a lattice in the function field setting is an (algebraic/holomorphic) vector bundle on $\mathbb P^1$. This works essentially because the sections of a vector bundle are a free module for the ring of functions on a curve. More precisely, a vector bundle on the projective line $\mathbb P^1$, restricted to any open subset, such as $\mathbb A^1$, produces a free module for the ring of functions on that subset. The extra data needed to extend the vector bundle to the missing point $\mathbb P^1 - \mathbb A^1$ can be viewed as a norm extending the local norm on the field of formal Laurent series at that point.  This local norm is exactly analogous to the symmetric bilinear form extending the local norm on the ``missing point" $\mathbb R$. 

In this analogy, the field of functions on the hyperelliptic curve $C$ corresponds to the imaginary quadratic field $K$. The class group $Cl(K)$ corresponds to the Jacobian of $C$.  This is a close correspondence - we can view the Jacobian as parameterizing holomorphic line bundles on $C$, and each ideal of the ring of functions on $C$ defines a holomorphic line bundle, with these line bundles isomorphic if and only if the ideals are equal up to multiplication by a meromorphic function.

To an ideal of $\mathcal O_K$, we may associate a lattice, by forgetting the action of $\mathcal O_K$ on it and remembering only the $\mathbb Z$-module structure as well as the symmetric bilinear form $(a,b) \mapsto \operatorname{Re}(a \overline{b})$ defined by the embedding into $\mathbb C$. A similar process, where we forget the action of a larger ring and remember only a smaller ring, occurs in algebraic geometry when we push forward a vector bundle along a finite morphism. In particular, for a line bundle $L$ on $C$, and $\pi: C\to \mathbb P^1$, the pushforward $\pi_* L$ defines an (algebraic/holomorphic) vector bundle of rank two on $\mathbb P^1$, whose fiber over a point at which $\pi$ is not branched is the sum of the fiber of $L$ over the two preimages of that point. For simplicity we assume that $L$ has degree $0$, in which case $\pi_* L$ has degree $-g-1$.

Thus, the analogue of $\mathbb C/I$ in the algebraic geometry setting is the map that takes a line bundle $L$ of degree $0$ to the vector bundle $\pi_* L$, and the analogue of the question of Michel and Venkatesh involves, for sets $U_1$ and $U_2$ of isomorphism classes of vector bundles on $\mathbb P^1$ (of rank $2$ and degree $-1-g$), the set \begin{equation}\label{vector-bundle-counting} \left\{ L \in J  \mid \pi_* L \in U_1, \pi_* (L \otimes L') \in U_2 \right\}. \end{equation} (Here, products of ideals correspond to tensor products of line bundles, which matches up with addition in the Jacobian.) 

How does this end up in the formulation we gave originally, involving $\Theta_a$ and $\Theta_b$, with not a single vector bundle in sight? This happens because there is a straightforward classification of vector bundles on $\mathbb P^1$, as sums of tensor powers of a fixed line bundle $\mathcal O(1)$. In particular, among vector bundles of fixed degree, the isomorphism class of a vector bundle $V$ is determined by the least $n$ such that $V \otimes \mathcal O(1)^{\otimes n}$ has a holomorphic global section. By the projection formula, $\pi_* L \otimes  \mathcal O(1)^{\otimes n}$ has a global section if and only if $L \otimes \pi^*  \mathcal O(1)^{\otimes n}$ has a global section. Such a global section would vanish at \[  \deg ( L \otimes \pi^*  \mathcal O(1)^{\otimes n}) = 2n\] points of $C$, and the class of $L$ could then be written as the sum of the images of those $2n$ points on the Abel-Jacobi map. So knowing whether $L$ lies inside $\Theta_a$ for all $a$, and in particular for $a=2n$, tells us the isomorphism class of $\pi_* L$. Using this, we can write the number of points in \eqref{vector-bundle-counting}, which we are interested in, in terms of the number of points in \eqref{theta-counting} that we discussed earlier.

Thus, the cohomology of \eqref{theta-counting} is the topological analogue of the equidistribution problem.

Here we only require \eqref{theta-counting} for even $a$ and $b$, but to handle line bundles $L$ of odd degree we would need to know odd $a$ and $b$ as well, and since the odd case is not any more difficult, we express \eqref{theta-counting} without the parity condition.

\subsection{Topological methods}

The spaces $\Theta_a \cap [x + \Theta_b]$ certainly match Desiderata \ref{intro-vc-desiderata}:

\begin{enumerate}

\item The varieties $\Theta_a$ already have singularities, although they are of a very special form. (The space $\operatorname{Sym}^a(C)$ is smooth and maps surjectively to $\Theta_a$, and one can control how bad the singularities are by studying this map. In particular, it is known that they are \emph{rational homology manifolds}.) Intersecting two of them at best preserves these singularities, and at worst introduces new ones, so the spaces $\Theta_a \cap [x + \Theta_b]$ are indeed very far from manifolds. 

\item The spaces $\Theta_a$ do not have nice descriptions in terms of equations as such. Even the Jacobian variety $J$ does not really have a reasonable description in terms of equations. However, $J$ certainly does have a nice algebraic-geometry structure, being an abelian variety. It turns out the spaces $\Theta_a$ have a nice description as intersections of translates of the ``theta divisor" $\Theta_{g-1}$, which carries over to $\Theta_a \cap [x + \Theta_b]$, and this can play the role of a description by equations. In particular, this fact was used in \cite[Corollary 3.3]{ST} to invoke the Lefschetz hyperplane theorem to control the low-degree cohomology of the spaces $\Theta_a \cap [x + \Theta_b]$,.

\item While $\Theta_a$ has a reasonably nice description as a configuration space of unordered $a$-tuples of points on $C$, which may collide, where two points in the same fiber of the map to $\mathbb P^1$ can annihilate each other, the intersections $\Theta_a \cap [x + \Theta_b]$ do not seem to have any such nice description.

\item The spaces $\Theta_a \cap [x + \Theta_b]$ lie in a nice family, parameterized by $x \in J$. For generic values of $x$, the singularities are no worse than in $\Theta_a$ and $\Theta_b$ separately.

\end{enumerate}

In \cite{ST}, Shende and Tsimerman use the Lefschetz hyperplane theorem to evaluate the low-degree cohomology of $\Theta_a \cap [x + \Theta_b]$. They then wish to use Poincar\'{e} duality to turn this into a description of the high-degree cohomology. For generic values of $x$, $\Theta_a \cap [x + \Theta_b]$ is a rational homology manifold and so Poincar\'{e} duality holds in every degree. Moreover, it turns out that for $x$ outside a special locus in $J$ that can be handled separately,   $\Theta_a \cap [x + \Theta_b]$  is a rational homology manifold outside a low-dimensional subset, which means that Poincar\'{e} duality is valid outside a few middle degrees. Thus, the problem of controlling the cohomology of $\Theta_a \cap [x + \Theta_b]$ is reduced to bounding the Betti numbers in these middle degrees.

The tool used is vanishing cycles theory in its guise as the \emph{characteristic cycle}. We can view the spaces $\Theta_a \cap [x + \Theta_b]$ as the fibers of a map $\Theta_a \times \Theta_b \to J$. Characteristic cycle theory starts when we compose this with a map $f: J \to \mathbb C$, which may be defined locally on an open set. We can study the vanishing cycles of the induced map $\Theta_a \times \Theta_b \to \mathbb C$. It turns out that, to a large extent, these vanishing cycles depend only on the derivative $df$ of $f$. The derivative is a $1$-form, which we can think of as a section of the cotangent bundle $T^* J$ of $J$. It turns out that there is a closed subset $\operatorname{CC} \subset T^*J $ such that local contributions to vanishing cycles only occur at the points where $df$ intersects $\operatorname{CC}$. Moreover, if these intersections are isolated, then the dimension of the vanishing cycles contribution is proportional to the intersection multiplicity of $df$ with $\operatorname{CC}$.

Now we are not in fact interested in any map $f: J \to \mathbb C$ in particular. Instead, we are interested in the fiber over a single point of $J$. However, Massey found a convenient way to derive from the characteristic cycle information on the cohomology of the fiber at one point \cite{Massey}. In particular, he gave bounds for the Betti numbers in each degree in terms of $\operatorname{CC}$. Shende and Tsimerman were able to compute the characteristic cycle by considering the convenient special case when $df$ is a $J$-invariant $1$-form on $J$, where everything can ultimately be reduced to calculations on $C$, and then used Massey's formula to produce suitable Betti number bounds.

The theory of the characteristic cycle, as well as Massey's bound, was purely characteristic $0$, and so Shende and Tsimerman were not able to solve the finite field point counting question, only its complex number analogue. However, Saito defined a good notion of the characteristic cycle in characteristic $p$ \cite{Saito}, and following this, \cite{me-massey} was able to prove a good analogue of the main result of \cite{Massey}, deducing an estimate for the number of points on  $\Theta_a \cap [x + \Theta_b]$ over finite fields.

\subsection{Future work}

An analogous question on the intersection \[\Theta_a \cap [x_1 + \Theta_{b_1} ] \cap [x_2 + \Theta_{b_2} ] + \dots\] would already be of interest, being analogous to the number-theoretic equidistribution problem on $X(1)^n$ for $n>2$. 

So would similar equidistribution problems involving maps $\pi: C \to \mathbb P^1$ of degree greater than $2$. One could either study the theta divisors, or more generally the locus where $\pi_* \mathcal O_C$ is isomorphic to a chosen vector bundle on $C$.

Finally, it would be interesting to study the problem on a base curve other than $\mathbb P^1$.


\begin{thebibliography}{99999} 
    
    
    \bibitem{AK}
   J.~C.~Andrade and J.~P.~Keating,
   Conjectures for the integral moments and ratios of {$L$}-functions over function fields,
   {\em Journal of Number Theory} {\bf 142} (2014), 102-148.
    
    \bibitem{sga4-3}
    M.~Artin, A.~Grothendieck, J.-L.~Verdier, eds. (1972).
    \textit{S\'{e}minaire de G\'{e}om\'{e}trie Alg\'{e}brique du Bois Marie - 1963-64 - Th\'{e}orie des topos et cohomologie \'{e}tale des sch\'{e}mas - (SGA 4) - vol. 3}
Lecture Notes in Mathematics 305.

\bibitem{BBDG}
A.~Beilionson, J.~Bernstein, P.~Deligne, O.~ Gabber, Fascieaux Pervers, {\em Ast\'{e}risque}  Soci\'{e}t\'{e} Math\'{e}matique de France, Paris. {\bf 100} (1982).

\bibitem{Burgess} D.~A.~Burgess, On character sums and primitive roots, Proc. London
Math. Soc. {\bf 3} (1962), 179-192.

\bibitem{Nonresonance} D.~C.~Cohen, A.~Dimka, and P.~Orlik, Nonresonance conditions for arrangements, {\em Annales de l'Institut Fourier}, {\bf 53} (2003), 1883-1896.

\bibitem{CEF}T.~Church, J.~S.~Ellenberg, and B.~Farb, Representation stability in cohomology and
asymptotics for families of varieties over finite fields, {\em Contemporary Mathematics} {\bf 620} (2014), 1-54.
    
    \bibitem{CFKRS}
 J.~B.~Conrey, D.~W.~Farmer, J.~P.~Keating, M.~O.~Rubinstein, and N.~C.~Snaith,
Integral moments of {$L$}-functions,
{\em Proceedings of the London Mathematical Society} {\bf 91} (2005), 33-104.



\bibitem{sga4h}
P.~Deligne, with the collaboration of J.~.F.~Boutot, A.~Grothendieck, L.~Illusie, and J.~L.~Verdier,
\textit{S\'{e}minaire de G\'{e}om\'{e}trie Alg\'{e}brique du Bous Marie SGA 4$\frac{1}{2}$ - Cohomologie Etale},
Lecture Notes in Mathematics 569.




\bibitem{sga7-ii} P.~Deligne, N.~Katz, eds. \emph{S\'{e}minaire de G\'{e}om\'{e}trie Alg\'{e}brique du Bois Marie - 1967-69 - Groupes de monodromie en g\'{e}om\'{e}trie alg\'{e}brique - (SGA 7) - vol. 2}, Lecture Notes in Mathematics (in French), Vol. 340. Springer-Verlag.


\bibitem{weil-ii}
P.~Deligne,
La conjecture de {W}eil: II,
{\em Publications mathematiques de l'I.H.\'{E}.S.} {\bf 52} (1980), 137-252.

\bibitem{Duke}
W.~Duke,
Hyperbolic distribution problems and half-integral weight {M}aass forms,
{\em Inventiones Mathematicae} {\bf 92} (1988), 73-90.

\bibitem{Ellenberg-AWS} J.~S.~Ellenberg, Arizona Winter School 2014 Course Notes: Geometric Analytic Number Theory, http://swc.math.arizona.edu/aws/2014/2014EllenbergNotes.pdf (2014).

\bibitem{EVW}
J.~Ellenberg, A.~Venkatesh, and C.~Westerland,
Homological stability for Hurwitz spaces and the Cohen-Lenstra conjecture over function fields,
{\em Annals of Mathematics} {\bf183} (2016), 729-786.

\bibitem{EWT}
J.~Ellenberg, C.~Westerland, and T.~Tran,
Fox-Neuwirth-Fuks cells, quantum shuffle algebras, and Malle's conjecture for function fields,
arXiv preprint: 1701.04541 (2017).

\bibitem{FreitagKiehl}
E.~Frietag and R.~Kiehl,
{\em Etale Cohomology and the Weil Conjecture},
A Series of Modern Surveys in Mathematics (1988)

\bibitem{Fu}
L.~Fu,
{\em Etale Cohomology Theory},
Nankai Tracts in Mathematics {\bf 13} (2015)



\bibitem{GoreskyMacpherson}
M.~Goresky and R.~Macpherson,
Intersection Homology Theory,
\emph{Topology} {\bf 19} (1980), 135-162.

\bibitem{sga7-i} A.~Grothendieck, ed. \emph{S\'{e}minaire de G\'{e}om\'{e}trie Alg\'{e}brique du Bois Marie - 1967-69 - Groupes de monodromie en g\'{e}om\'{e}trie alg\'{e}brique - (SGA 7) - vol. 2}, Lecture Notes in Mathematics (in French), Vol. 288. Springer-Verlag.

\bibitem{HastMatei}
D.~R.~Hast and V.~Matei,
Higher Moments of Arithmetic Functions in Short Intervals: A Geometric Perspective,
{\em IMRN} (2018), rnx310.

\bibitem{Hoo91} C. Hooley, \emph{On the number of points on a complete intersection over a finite field}, J. of Number Theory, 38.3 (1991), 338-358.

\bibitem{katzbetti} N.~M.~Katz, Sums of Betti Numbers in Arbitrary Characteristic, {\em Finite Fields and their Applications}, {\bf 7} (2001), 29-44.

\bibitem{Katz-PDC}
N.~M.~Katz,
On a question of Keating and Rudnick about primitive Dirichlet characters with squarefree conductor,
{\em IMRN} (2013), 3221-3249.

\bibitem{Khayutin}
I.~Khayutin,
Joint equidistribution of {CM} points,
{\em Annals of Mathematics}, {\bf 189} (2019), 145-276.

\bibitem{Lang}
S.~Lang,
Algebraic groups over finite fields,
{\em American Journal of Mathematics}, {\bf 78} (1956), 555-563.


\bibitem{Massey}
D.~Massey,
Numerical invariants of perverse sheaves,
{\em Duke Mathematical Journal} {\bf 73} (1994), 307-369.

\bibitem{MasseyIntro}
D.~Massey,
Notes on perverse sheaves and vanishing cycles,
arXiv preprint: 9908107 (199).

\bibitem{MV}
P.~Michel and A.~Venkatesh,
Equidistribution, {$L$}-functions and Ergodic theory: on some problems of {Y}u. {V}. {L}innik (unpublished version), http://math.stanford.edu/~akshay/research/linnik.pdf (2006).


\bibitem{Polya}
G.~P\'{o}lya, Ueber die Verteilung der quadratischen Reste und Nichtreste, Nachr. Akad. Wiss. Goettingen (1918), 21–29.

\bibitem{Saito}
T.~Saito, The characteristic cycle and the singular support of a constructible sheaf, {\em Inventiones Mathematicae}, {\bf 207} (2017), 597-695.

%
%
%
%



\bibitem{square-root}
W.~Sawin,
Square-root cancellation for sums of factorization functions over short intervals in function fields
arXiv preprint: 1809.015137 (2018).

\bibitem{me-massey}
W.~Sawin,
Bounds for the stalks of perverse sheaves in characteristic p and a conjecture of Shende and Tsimerman (with an appendix by J.~Tisimerman), arXiv preprint: 1907.04850 (2019).

\bibitem{square-free}
W.~Sawin,
Square-root cancellation for sums of arithmetic functions in squarefree progressions over function fields
(in preparation).

\bibitem{SS}
W.~Sawin and M.~Shusterman,
On the Chowla and twin primes conjectures over $\mathbb F_q[T]$,
arXiv preprint: 1808.04001 (2019).

\bibitem{ST}
V.~Shende and J.~Tsimerman,
Equidistribution in {$\operatorname{Bun}_2(\mathbb P^1)$},
{\em Duke Mathematical Journal} {\bf 166} (2017), 3461-3504.

\bibitem{Vinogradov}
I.~M.~Vinogradov,
Sur la distribution des residus and nonresidus des puissances,{\em J. Soc. Phys. Math. Univ. Permi} (2018), 18-28.

\bibitem{Weil}
A.~Weil,
On some exponential sums, {\em Proceedings of the National Academy of Sciences} {\bf 34} (1948), 204-207.



\end{thebibliography}
\end{document}